\documentclass[reqno]{amsart} 

\usepackage[utf8]{inputenc} 

\usepackage[foot]{amsaddr} 
\usepackage{verbatim} 
\usepackage{amssymb}
\usepackage{amsthm}
\usepackage{amsmath}
\usepackage{amsfonts}
\usepackage{latexsym}
\usepackage{mathtools} 
\usepackage{enumitem} 
\usepackage[english]{babel} 

\usepackage{esint}

\usepackage{calrsfs}

\usepackage[usenames,dvipsnames]{xcolor} 

\definecolor{dpblue}{RGB}{1,31,91} 
\usepackage[colorlinks=true, pdfstartview=FitV, linkcolor=dpblue, citecolor=dpblue, urlcolor=dpblue]{hyperref}     

\theoremstyle{plain}
\newtheorem{theorem}{Theorem}
\newtheorem{lemma}[theorem]{Lemma}

\newtheorem{proposition}[theorem]{Proposition}
\newtheorem{remark}[theorem]{Remark}

\newtheorem{corollary}[theorem]{Corollary}

\numberwithin{equation}{section}
\numberwithin{theorem}{section}

\allowdisplaybreaks

\newcommand{\eqdef }{\overset{\mbox{\tiny{def}}}{=}}

\newcommand{\pv}{p}
\newcommand{\pZ}{\pv^0}

\newcommand{\qv}{q}
\newcommand{\qZ }{\qv^0}

\newcommand{\pqs}{{p^0}+{q^0}+\sqrt{s}}
\newcommand{\sqs}{\sqrt{s}}
\newcommand{\tet}{\frac{\theta}{2}}
\newcommand{\pqhat}{\frac{1}{2g}\left(\frac{{q^0}}{{p^0}}d-a\right)}
\newcommand{\CE}{\frac{\left(p+q\right)\cdot w }{2g\left({p^0}+{q^0}+\sqrt{s}\right)^2s^{\frac{3}{2}}}}
\newcommand{\aoo}{a_{11}}
\newcommand{\aot}{a_{12}}
\newcommand{\aox}{a_{13}}
\newcommand{\ato}{a_{21}}
\newcommand{\att}{a_{22}}
\newcommand{\atx}{a_{23}}
\newcommand{\axo}{a_{31}}
\newcommand{\axt}{a_{32}}
\newcommand{\axx}{a_{33}}

\newcommand{\R}{\mathbb{R}}

\title[On the Determinant Problem for the Relativistic Collision map]{On the Determinant Problem for the Relativistic Boltzmann Equation}

\author[J. Chapman]{James Chapman$^\ddagger$}
\address{$^\ddagger$Department of Mathematics, University of Pennsylvania, Philadelphia, PA 19104, USA.  \href{mailto:chapmanj@sas.upenn.edu}{chapmanj@sas.upenn.edu}}

\author[J. W. Jang]{Jin Woo Jang$^*$}
\address{$^*$Institute for Applied Mathematics, University of Bonn, 53115 Bonn, Germany. \href{mailto:jangjinw@iam.uni-bonn.de}{jangjinw@iam.uni-bonn.de} }

\author[R. M. Strain]{Robert M. Strain$^\dagger$}
\address{$^\dagger$Department of Mathematics, University of Pennsylvania, Philadelphia, PA 19104, USA.  \href{mailto:strain@math.upenn.edu}{strain@math.upenn.edu}}

\AtBeginDocument{%
   \def\MR#1{}
}

\begin{document}




\begin{abstract}
This article considers a long-outstanding open question regarding the Jacobian determinant for the relativistic Boltzmann equation in the \textit{center-of-momentum} coordinates. For the Newtonian Boltzmann equation, the center-of-momentum coordinates have played a large role in the study of the Newtonian non-cutoff Boltzmann equation, in particular we mention the widely used cancellation lemma \cite{ADVW}. In this article we calculate specifically the very complicated Jacobian determinant, in ten variables, for the relativistic collision map from the momentum $p$ to the post collisional momentum $p'$; specifically we calculate the determinant for $p\mapsto u = \theta p'+\left(1-\theta\right)p$ for $\theta \in [0,1]$. Afterwards we give an upper-bound for this determinant that has no singularity in both $p$ and $q$ variables. Next we give an example where we prove that the Jacobian goes to zero in a specific pointwise limit. We further explain  the results of our  numerical study which shows that the Jacobian determinant has a very large number of distinct points at which it is machine zero. This generalizes the work of Glassey-Strauss (1991) \cite{MR1105532} and Guo-Strain (2012) \cite{Guo-Strain2}. These conclusions make it difficult to envision a direct relativistic analog of the Newtonian cancellation lemma in the center-of-momentum coordinates.
\end{abstract}

\setcounter{tocdepth}{1}

\maketitle
\tableofcontents

\thispagestyle{empty}

\section{Introduction}\label{sec:intro}

The special relativistic Boltzmann equation is a fundamental model for relativistic gases \cite{C-K, DeGroot} which obey Einstein's theory of special relativity. The equation describes the dynamics of the statistical distribution of relativistic particles when the binary collisions among particles occur frequently enough to dominate the dynamics, so that one can assume that the rate of change along particle paths in phase space is mainly due to the binary collisions among the particles. 
The relativistic Boltzmann equation is a central model in the relativistic collisional kinetic theory.

The relativistic Boltzmann equation can be expressed as 
$$
\partial_t F+\hat{p}\cdot \nabla_x F= Q\left(F,F\right),
$$
and the normalized velocity of a particle $\hat{p}$ is given by
$$
\hat{p}=c\frac{p}{{p^0}}=\frac{p}{\sqrt{1+\frac{|p|^2}{c^2}}}.
$$
Above $p\in \mathbb{R}^3$, $x\in \Omega$ where $\Omega$ is a domain and $t\ge 0$.
Here $c$ denotes the speed of light, which is a constant.  Also, $p^0 = \sqrt{c^2+|p|^2}$ denotes the relativistic particle energy with the rest mass normalized to be 1.  For $q\in \mathbb{R}^3$, then $q^0$ is defined similarly.
From here on we normalize the speed  of light to one by setting $c=1$.

The relativistic Boltzmann collision operator is given by
\begin{equation}
	\label{omegaint}
	Q\left(f,h\right)=\int_{\mathbb{R}^3} dq\int_{\mathbb{S}^2} dw \hspace{1mm} v_\phi
	\hspace{1mm} 
	\sigma\left(g,\vartheta\right)[f\left(p'\right)h\left(q'\right)-f\left(p\right)h\left(q\right)],
\end{equation}
In this operator we consider a pair of relativistic particles with momenta $p$ and $q$ that after a collision have post-collisional momenta $p'$ and $q'$. The post-collisional momenta $p'$ and $q'$ can further be written as \eqref{p'} and \eqref{q'} below.
Then $v_\phi=v_\phi\left(p,q\right)$ is the M$\phi$ller velocity which is given by
$$
v_\phi\left(p,q\right)
\eqdef 
\sqrt{\Big|\frac{p}{{p^0}}-\frac{q}{{q^0}}\Big|^2-\Big|\frac{p}{{p^0}}\times\frac{q}{{q^0}}\Big|^2}=\frac{g\sqrt{s}}{{p^0}{q^0}}.
$$
Above $g$ and $s$ are defined below in \eqref{g2} and \eqref{s} respectively.

Further the relativistic Boltzmann collision kernel $\sigma(g,\vartheta)$ is a non-negative function which only depends on the relative momentum $g$ and the scattering angle 
$\vartheta$. 
The scattering angle $\vartheta$ is defined by
\begin{equation}\label{scattering.angle}
\cos\vartheta\eqdef \frac{k}{|k|}\cdot w,
\end{equation}
where $k$ is defined as 
$$
k\eqdef -\frac{p+q}{\sqrt{s}}(p^0-q^0)+(p-q)+(\gamma-1)(p+q)\frac{(p+q)\cdot(p-q)}{|p+q|^2}.
$$ 
And $\gamma$ is defined as
$$
\gamma \eqdef \frac{p^0+q^0}{\sqrt{s}}.
$$
The proof for this identity is given in \cite[page 5-6]{MR2765751}.
This angle $\vartheta$ was proven to be a well defined angle in \cite{GL1996}.  
It is standard to assume that $\sigma$ takes the form of the product in its arguments; i.e., $$\sigma(g,\vartheta)\eqdef \Phi(g)\sigma_0(\vartheta).$$ In general, we suppose both $\Phi$ and $\sigma_0$ are non-negative functions.

Depending on the local integrability of the angular function $\vartheta \mapsto \sigma_0(\vartheta)$, we classify the problem into two regimes: with and without an angular \textit{cutoff}. If the angular function satisfies either $\vartheta \mapsto\sigma_0\in L^1_{loc}(\mathbb{S}^2)$ or $\sigma_0 \in L^\infty (\mathbb{S}^2)$, then we say that the problem is with an angular \textit{cutoff} \cite{Gra}. Otherwise, we say that
the problem is without an angular \textit{cutoff}. Examples of physical non-cutoff relativistic collision kernels were explained for example in \cite{MR2679588,Jang2016}.

Without loss of generality, we may assume that the collision kernel $\sigma$ is supported on
\begin{equation}\label{angle.condition}
    \cos\vartheta\geq 0, \quad  \text{i.e.} \quad 0\leq \vartheta \leq \frac{\pi}{2}.
\end{equation}
Otherwise, the following \textit{symmetrization} \cite{MR1379589} will reduce to this case:
 $$
\bar{\sigma}(g,\vartheta)=[\sigma(g,\vartheta)+\sigma(g,-\vartheta)]1_{\cos\vartheta\geq 0},
$$
where $1_A$ is the indicator function of the set $A$.

The post-collisional momenta in the \textit{center-of-momentum} expression are written as 
\begin{equation}
	\label{p'}
	p'=\frac{p+q}{2}+\frac{g}{2}\left( w +\left(\gamma-1\right)\left(p+q\right)\frac{\left(p+q\right)\cdot w }{|p+q|^2}\right),
\end{equation}
and
\begin{equation}
	\label{q'}
	q'=\frac{p+q}{2}-\frac{g}{2}\left( w +\left(\gamma-1\right)\left(p+q\right)\frac{\left(p+q\right)\cdot w }{|p+q|^2}\right).
\end{equation}
We point out that $\gamma-1 \geq 0$ from \eqref{gamma.calc} below.
Note that $p, q \in \mathbb{R}^3$ and $| w| =1$ is on the sphere $w \in \mathbb{S}^2$.  

Further the energy-momentum conservation laws say that
\begin{equation}\label{collision.invariants}
	p+q = p^\prime + q^\prime, \quad 
p^0 + q^0 = p^{\prime 0} + q^{\prime 0}.    
\end{equation}
Now the quantities $s$ and $g$ denote the square of the \textit{total energy} in the \textit{center-of-momentum} system $p+q=0$ and the \textit{relative momentum}, respectively. They are defined as
\begin{equation}
	\label{g2}
	g=g(p,q)=\sqrt{2\left({p^0}{q^0}-\sum^3_{i=1} p^i q^i-1\right)}\geq0,
\end{equation}
and 
\begin{equation}
	\label{s}
	s=s(p,q)=g^2+4 = 2\left({p^0}{q^0}-\sum^3_{i=1} p^i q^i+1\right),
\end{equation}
We also notice that $s\geq 4.$

For the Newtonian Boltzmann equation the coordinate system which is analogous to \eqref{p'} and \eqref{q'} is the following:
\begin{equation}\label{newtonian.pq}
    p' = \frac{p+q}{2}+\frac{|p-q|}{2}w , \quad q' = \frac{p+q}{2}-\frac{|p-q|}{2}w.
\end{equation}
Indeed, taking an appropirate limit as $c\to\infty$ in \eqref{p'} and \eqref{q'} yields \eqref{newtonian.pq}.  In the non-cutoff Newtonian Boltzmann theory the change of variables $p\to p'$ using \eqref{newtonian.pq}, and in particular the cancellation lemma from \cite{ADVW}, has been shown to be very important for understanding the fractional diffusive behavior of the collision operator.  In particular with \eqref{newtonian.pq} for 
\begin{equation}
	\label{u.split}
	u = \theta p'+\left(1-\theta\right)p, \quad \theta \in [0,1].
\end{equation}
The change of variable $p\to u$ is known to have Jacobian determinant \cite{MR2784329,MR2795331}:
$$
\left| \frac{d u_i}{ dp_j}  \right| 
= 
\left(1-\frac{\theta}{2}\right)^2\left\{
\left(1-\frac{\theta}{2}\right) + 
\frac{\theta}{2} \langle k, w \rangle
\right\}.
$$
Where the unit vector is $k = (p-q)/|p-q|$.  Therefore under the condition that $\langle k, w \rangle\ge 0$ then this Jacobian is uniformly bounded from below.

In this article we unfortunately notice that the analogous change of variables $p\mapsto u$ in the relativistic problem using \eqref{p'} behaves in comparison very badly, and can have zero determinant even under the corresponding angle condition \eqref{angle.condition}.

\subsection{A problem with the Jacobian determinant}
The rest of this article mainly deals with a long-outstanding open question regarding the Jacobian determinant, which arises when one takes a change of variables from a pre-collisional momentum $p$ or $q$ to a post-collisional momentum $p'$ or $q'$, in the \textit{center-of-momentum} coordinates. 

Historically, the well-posedness theory for the classical and the relativistic Boltzmann equations have been studied quite extensively. 
One of the main difficulties which arise in the theory of well-posedness for the Boltzmann equation is to obtain an appropriate a-priori estimate. In other words, one must treat the gain and the loss term appropriately, so one can obtain some desired estimates on them.

 Whenever one deals with estimating the upper- or lower-bounds for the Boltzmann collision operator \eqref{omegaint}, especially for the gain term in the operator, one encounters the Jacobian determinant as the functions in the post-collisional momentum $p'$ or $q'$ appear inside the integration with respect to the pre-collisional measures $dp$ or $dq.$ However, it unfortunately appears to be of limited utililty to use the change of pre-post collisional variables $p \mapsto p'$ or $q\mapsto q'$ as we will explain how the Jacobian is no longer uniformly bounded above and below in the relativistic scenario. 
 
 Traditionally, this issue has been resolved along the following different lines:
 \begin{itemize}
 \item One approach is to check if the Jacobian of the change of variables $p\ (\text{or} \ q)\mapsto p'\ (\text{or} \ q')$ is uniformly bounded above and below. The situation that one must consider is the change of variables in only one variable like $(p,q)\mapsto (p',q) \ (\text{or} \ \mapsto (p,q'))$.  This occurs especially when one considers the linearization of the Boltzmann collision operator \eqref{omegaint}.  This, indeed, is useful in the non-relativistic scenario with the center-of-momentum representation of the variables \eqref{newtonian.pq}.  In this Newtonian situation it is known, e.g. \cite{ADVW}, that the Jacobian is well behaved.

 \item Another approach is to check if one can also change both variables $(p,q)$ to $(p',q')$ at the same time.   In the relativistic situation for this approach one can use the following coordinates \cite{GS3}: \begin{equation}\label{alternativerepresentation}
 p'=p+a(p,q,w)w,\quad \text{and}\quad q'-a(p,q,w)w,\end{equation} where 
 $$
 a(p,q,w)=\frac{2(p^0+q^0)(w\cdot (p^0q-q^0p))}{(p^0+q^0)^2-(w\cdot[p+q])^2}.
 $$ 
 These post-collisional coordinates are the relativistic analog of the following Newtonian post-collisional variables
 $$
 p' = p -((p-q)\cdot w ) w, \quad  q' = q +((p-q)\cdot w ) w.
 $$
 In this case, the change of variables $(p,q)$ to $(p',q')$ does not really result in any harm because we have the following: 
 \begin{equation}\notag
 \begin{split}\left|\frac{\partial(p,q)}{\partial(p',q')}\right|&=1,\ \text{in the non-relativistic case},\\
 &=\frac{p^0q^0}{p'^0q'^0},\ \text{in the relativistic case \cite{MR1105532}} .
 \end{split}
 \end{equation}

 \item A third approach is to use the Carleman representation of the collision operator; mainly, one derives and uses an \textit{alternative} representation of the collision operator. However, this is still not always easy to follow in the relativistic case as discussed in \cite{Jang2016} and \cite{1907.05784}.
 \end{itemize}

For these approaches, the remaining possible strategies for performing an appropriate change of variables for the relativistic Boltzmann collision operator is  to either consider the representation \eqref{alternativerepresentation}, or to go through deriving an appropriate Calreman-type representation for the collision operator and try to estimate them using this representation.   The variables \eqref{alternativerepresentation} have a disadvantage, as the upper-bound for the Jacobian has huge momentum growth in $p$ and $q$ variables as
$$
|\nabla_q p'_i|+|\nabla_q q'_i|\lesssim (p^0)^5 q^0.
$$ 
Although the growth in the $q$ variable can be treated with a compensating exponential decay in the $q$ variable if one takes the standard symmetric linearization around a relativsitic Maxwellian as in for instance \cite{GS3,GS4,Guo-Strain2,MR2728733}, the growth in the $p$ variable is still problematic; this difficulty was studied in \cite{Guo-Strain2}. The latter method of deriving and using a Carleman-type representation to change variables is also difficult because it contains the estimates on an unbounded non-flat hypersurface, as observed in \cite{Jang2016,1907.05784}.

Therefore, in \cite{Guo-Strain2}, the authors  used the \textit{center-of-momentum} \eqref{p'} representation away from the singular region; they used the fact that the post-collisional variables $p'$ and $q'$ in the \textit{center-of-momentum} representation has its singularity (i.e., the Jacobian of $\left|\frac{\partial^k p'}{\partial^k p}\right|$ vanishes) when $p-q=0$ if $k\geq 1$ and $p+q=0$ if $k\geq 2$. The authors provided an upper-bound estimate for the Jacobian away from the singularities and have shown that the Jacobian does not have a growth in the $p$ variable away from the singularities. More precisely, what they have computed is the bound for the Jacobian in the region away from the singularities, if $|p|\geq 1$ and $|p|^{1/m}>2q^0$ for some integer $m\ge 1$ then they have shown that
$$
|\partial_\beta p'|+|\partial_\beta q'|\lesssim (q^0)^n,
$$ 
for some integer $n\geq 1$ which depends upon $\beta\neq 0$. Here $\partial_\beta$ is the multi-index notation for the derivatives with respect to the $p$ variable as follows: $\beta=[\beta^1,\beta^2,\beta^3]$ 
and
$\partial_\beta=\partial_{p_1}^{\beta^1} \partial_{p_2}^{\beta^2}\partial_{p_3}^{\beta^3}$. So, we can say that the use of the \textit{center-of-momentum} coordinates has its own advantage that it does not show any growth in the $p$ variable away from the singularity; this is explained in Lemma 3 of \cite{Guo-Strain2}. 

In the non-cutoff scenario, when the angular \textit{cutoff} assumption is removed, the situation is even worse as one must utilize the cancellations from the gain and loss operators to implement cancellations of the high angular singularities. Therefore, one must obtain the upper- and the lower- bounds for the Jacobian of the change of variables $\left(p,q\right) \rightarrow \left(u,q\right)$ where $u$ is defined as \eqref{u.split}.  In this situation, we observe numerically in this work below that the the zeros of the Jacobian $\left|\frac{\partial p'}{\partial p}\right|$ (which is a singularity for $\left|\frac{\partial p}{\partial p'}\right|$) occur in large regions depending on all $(\theta, p,q,w)$.

In this paper, we compute the Jacobian determinant in \eqref{det.up} for the change of variables from $p$ to $u$  for the noncutoff Boltzmann theory in the \textit{center-of-momentum} representation \eqref{p'} even on the singular region.  We calculate a very explicit expression for the Jacobian and provide its upper-bound that has no singularity in the $p$ and $q$ variables. This generalizes the work in \cite{Guo-Strain2} away from singularities, and the work in \cite{Jang2016}.  We will further prove that the Jacobian determinant can go to zero in a limit.  And we explain numerical evidence that the Jacobian \eqref{det.up} has a large number of distinct values where it is machine zero to up to two hundred digits of precision.

\subsection{Outline of the paper}
In the next Section \ref{sec:upperBD} we calculate the Jacobian determinant \eqref{det.up} for the the change of variables from $p \mapsto u$ in \eqref{u.split}.  We also prove the upper bound for this Jacobian.  Then in Section \ref{sec:lowerBD} we prove that the lower bound of the Jacobian is zero.  In Section \ref{sec:numerical} we present numerical evidence that the Jacobian determinant has a large number of distinct values which make it machine zero.  Then, lastly, in Appendix \ref{sec:ALTjacobian} we give an alternative expression for the determinant\eqref{det.up} in Proposition \ref{Jaco2}.

%
%
%
%

	\section{The upper-bound of the Jacobian of the collision map}\label{sec:upperBD}
	We consider a pair of relativistic particles with momenta $p$ and $q$ that collide and diverge with post-collisional momenta $p'$ and $q'$. Using the \textit{center-of-momentum} expressions, we can represent the post-collisional variables $p'$ and $q'$ as \eqref{p'} and \eqref{q'}. In this section, we are interested in the Jacobian of the collision map $\left(p,q\right) \rightarrow \left(u,q\right)$ where $u$ is defined as $u\eqdef \theta p'+\left(1-\theta\right)p$ in \eqref{u.split} for some $\theta\in\left(0,1\right)$. The Jacobian will be computed explicitly and it will be shown that the Jacobian is bounded above in the variable $p$ and $q$. 
	
	This section is in particular devoted to estimate the upper-bound of
	$
	\det\left( \frac{\partial u}{\partial p}\right).
	$
Recall that the post-collisional momentum in the center-of-momentum expression is defined as \eqref{p'} 
	where 
	\begin{equation}\label{gamma.calc}
	\gamma-1\eqdef \frac{p^0+q^0-\sqrt{s}}{\sqrt{s}}=
	\frac{(p^0+q^0)^2-s}{\sqrt{s}\left(p^0+q^0+\sqrt{s}\right)}=\frac{|p+q|^2}{\sqrt{s}\left(p^0+q^0+\sqrt{s}\right)}.
\end{equation}
We will use the calculation in \eqref{gamma.calc} rather frequently in the proofs below. Notice also that $\gamma-1\geq0$ from \eqref{gamma.calc}.
	
	We now state our main theorem:
	\begin{theorem}
		\label{Jacobian}
The Jacobian determinant $\det\left( \frac{\partial u}{\partial p}\right)$ is equal to
\begin{equation}\label{det.up}
\det\left( \frac{\partial u}{\partial p}\right)=A^3+P_2A^2+P_3A
\end{equation}
where $A\in\left(1-\theta,1\right)$ is defined as 
\begin{equation}\label{def.A}
A\eqdef \left(1-\frac{\theta}{2}\right)+\frac{\theta}{2}\left(g\frac{\left(\gamma-1\right)\left(p+q\right)\cdot w }{|p+q|^2}\right),
\end{equation}
and $P_2$ and $P_3$ are defined as in \eqref{P2} and \eqref{P3} below.  They satisfy
\begin{equation}\label{def.P2}
|P_2|\lesssim (q^0)^{\frac{3}{2}}\left(1+\frac{\sqrt{p^0}}{s}\right), 
\end{equation}
and
\begin{equation}\label{def.P3}
	|P_3|\lesssim \frac{q^0}{s}.
\end{equation}
	\end{theorem}
	Since $A\in \left(1-\theta,1\right)$, we obtain the following corollary on the upper-bound for the derivative of the collision map:
	\begin{corollary}
		The Jacobian determinant $\det\left( \frac{\partial u}{\partial p}\right)$ is bounded above as
		$$
		\left|\det\left( \frac{\partial u}{\partial p}\right)\right|\lesssim (p^0)^{\frac{1}{2}}(q^0)^{\frac{3}{2}}.
		$$
	\end{corollary} 
	
	\begin{remark}
	Here we remark that our estimate on the Jacobian is the first result which does not contain any singularity in $p$ and $q$ variables in the use of the \textit{center-of-momentum} coordinates. 
	A similar work on the relativistic Jacobian has been done by Glassey and Strauss \cite{MR1105532} in 1991 with the use of an alternative representation of the post-collisional momenta \eqref{alternativerepresentation}. More precisely, they proved that 
$$|\nabla_q p'_i|+|\nabla_q q'_i|\lesssim (p^0)^5 q^0.$$
Further, as we discussed in the previous section, the use of the variables \eqref{alternativerepresentation} creates a growth in $q$ variable that can cause severe difficulties. One can remove the growth in $|q|$ by averaging in $w$ variable as $$
	\sum_{i,j}\int_{\mathbb{S}^2}  \left\{ \left|\frac{\partial p'_i}{\partial q_j}\right|+\left|\frac{\partial q'_i}{\partial q_j}\right|\right\}dw \lesssim \left(p^0\right)^5.
	$$   
	Note that the growth in $|p|$ is not removed.
	This is proven in \cite[Theorem 2]{MR1105532}.
	\end{remark}
We will frequently use the following well known coercive inequality for the relative momentum in the center of momentum framework.

\begin{lemma}[Lemma 3.1 (i) on page 316 of \cite{GS3}]\label{coersive inequality}  The relative momentum $g$ satisfies the following inequalities:
\begin{equation}\label{gINEQ}
\frac{|p-q|}{\sqrt{\pZ \qZ }}\leq g(p,q)\leq |p-q|.
\end{equation}
\end{lemma}
We remark that in \cite{GS3}, the notation $g$ is used for $\frac{1}{2}g$ from \eqref{g2}; this would change the constant in the upper and lower bound of \eqref{gINEQ} by two.  A proof of  \eqref{gINEQ} can also be found in \cite[Proposition 3.1]{1903.05301}.

We now give a brief outline of the proof of Theorem \ref{Jacobian}.    We first take a derivative of $u_i$ in \eqref{u.split} with respect to $p_j$ and decompose the derivative $\frac{\partial u_i}{\partial p_j}$ into a linear combination of the elements of the tensor product of $(p_i,q_i,w_i)^\top$ with $(p_j,q_j,w_j)^\top$. In order to obtain the Jacobian determinant, we define the orthonormal basis $\{w,\bar{w},\tilde{w}\}$ of $\mathbb{R}^3$ and further represent the derivative $\frac{\partial u_i}{\partial p_j}$ as a linear combination of the elements of the tensor product of $(w_i,\bar{w}_i,\tilde{w}_i)^\top$ and $(w_j,\bar{w}_j,\tilde{w}_j)^\top$.  The main difficulty in this proof is to choose carefully several very complicated row and column reductions.  After that we are able to represent the Jacobian determinant as a cubic polynomial with respect to the quantity $A$ in \eqref{det.up}. Then we estimate the upper-bounds for each coefficient of the polynomial, and we further use those to obtain the upper-bounds of whole the Jacobian determinant.

	\begin{proof}[Proof for Theorem \ref{Jacobian}]
		The post-collisional momenta in the \textit{center-of-momentum} expression are written as \eqref{p'} and \eqref{q'}.
		We further recall \eqref{gamma.calc},
		so that we also have
		\begin{equation}
			\label{p''}
			p'=\frac{p+q}{2}+\frac{g}{2}\left( w +\left(p+q\right)\frac{\left(p+q\right)\cdot w }{\sqrt{s}(p^0+q^0+\sqrt{s})}\right),
		\end{equation}
			Now we compute the derivative
			$$
			\frac{\partial u_i}{\partial p_j}=\left(1-\theta\right)\delta_{ij}+\theta\frac{\partial p_i'}{\partial p_j},
			$$
			for any choices of $i, j\in\{1,2,3\}.$
			
			 By \eqref{p'} and \eqref{p''} we have
			\begin{align*}
			\frac{\partial p_i'}{\partial p_j}&=\frac{1}{2}\bigg(\delta_{ij}+\frac{\partial g}{\partial p_j} w _i+\frac{\partial g}{\partial p_j}\left(\gamma-1\right)\left(p_i+q_i\right)\frac{\left(p+q\right)\cdot w }{|p+q|^2}\\&+g\frac{\left(p+q\right)\cdot w }{\sqrt{s}\left(p^0+q^0+\sqrt{s}\right)}\delta_{ij}+g\left(p_i+q_i\right)\frac{\partial}{\partial p_j}\left(\frac{\left(p+q\right)\cdot w }{\sqrt{s}\left(p^0+q^0+\sqrt{s}\right)}\right)\bigg).
			\end{align*}
			Then from \eqref{g2} we have that 
			\begin{align*}
			\frac{\partial g}{\partial p_j}
			&=
			\frac{\partial}{\partial p_j}\left(\sqrt{-\left(p^0-q^0\right)^2+|p-q|^2}\right)
			\\
			&=
			\frac{1}{2g}\frac{\partial}{\partial p_j}\left(-\left(p^0-q^0\right)^2+|p-q|^2\right)
			\\
			&=
			\frac{1}{2g}\left(-2\left(p^0-q^0\right)\frac{\partial p^0}{\partial p_j}+2|p-q|\frac{\partial |p-q|}{\partial p_j}\right)
			\\
			&=
			\frac{1}{g}\left(-\left(p^0-q^0\right)\frac{p_j}{p^0}+\left(p_j-q_j\right)\right)
		=
			\frac{1}{g}\left(\frac{q^0}{p^0}p_j-q_j\right).
			\end{align*}
			Also, we have
			\begin{align*}
			&\frac{\partial}{\partial p_j}\left(\frac{\left(p+q\right)\cdot w }{\sqrt{s}\left(p^0+q^0+\sqrt{s}\right)}\right)
			\\
			&=
			\frac{ w _j\left(\sqrt{s}\left(p^0+q^0+\sqrt{s}\right)\right)-\left(p+q\right)\cdot w \frac{\partial}{\partial p_j}\left(\sqrt{s}\left(p^0+q^0+\sqrt{s}\right)\right)}{s\left(p^0+q^0+\sqrt{s}\right)^2}.
			\end{align*}
			Note that we have
			\begin{align*}
			\frac{\partial \sqrt{s}}{\partial p_j}
			&=
			\frac{\partial}{\partial p_j}\left(\sqrt{\left(p^0+q^0\right)^2-|p+q|^2}\right)
			\\
		&=
			\frac{1}{2\sqrt{s}}\frac{\partial}{\partial p_j}\left(\left(p^0+q^0\right)^2-|p+q|^2\right)
			\\
			&=
			\frac{1}{2\sqrt{s}}\left(2\left(p^0+q^0\right)\frac{\partial p^0}{\partial p_j}-2|p+q|\frac{\partial |p+q|}{\partial p_j}\right)
			\\
			&=
			\frac{1}{\sqrt{s}}\left(\left(p^0+q^0\right)\frac{p_j}{p^0}-\left(p_j+q_j\right)\right)
			=
			\frac{1}{\sqrt{s}}\left(\frac{q^0}{p^0}p_j-q_j\right).
			\end{align*}
			Then we obtain that
			\begin{align*}
			\frac{\partial}{\partial p_j}\left(\sqrt{s}\left(p^0+q^0+\sqrt{s}\right)\right)
			&=
			\frac{\partial \sqrt{s}}{\partial p_j}\left(p^0+q^0+\sqrt{s}\right)+\sqrt{s}\left(\frac{\partial p^0}{\partial p_j}+\frac{\partial \sqrt{s}}{\partial p_j}\right)
			\\
			&=\left(\frac{q^0}{p^0}p_j-q_j\right)\frac{p^0+q^0+2\sqrt{s}}{\sqrt{s}}+\frac{\sqrt{s}}{p^0}p_j.
			\end{align*}
			Therefore, combining the calculations above we have 
			\begin{align*}
			\frac{\partial p_i'}{\partial p_j}&=\frac{1}{2}\bigg(\delta_{ij}+\frac{1}{g}\left(\frac{q^0}{p^0}p_j-q_j\right)\left(w_i+(\gamma-1)\left(p_i+q_i\right)\frac{\left(p+q\right)\cdot w }{|p+q|^2}\right)\\&+g\frac{\left(p+q\right)\cdot w }{\sqrt{s}\left(p^0+q^0+\sqrt{s}\right)}\delta_{ij}+g\left(p_i+q_i\right)\frac{\partial}{\partial p_j}\left(\frac{\left(p+q\right)\cdot w }{\sqrt{s}\left(p^0+q^0+\sqrt{s}\right)}\right)\bigg)\\
			&=\frac{1}{2}\bigg(\delta_{ij}+\frac{1}{g}\left(\frac{q^0}{p^0}p_j-q_j\right)\left(w_i+(\gamma-1)\left(p_i+q_i\right)\frac{\left(p+q\right)\cdot w }{|p+q|^2}\right)\\&+g\frac{\left(p+q\right)\cdot w }{\sqrt{s}\left(p^0+q^0+\sqrt{s}\right)}\delta_{ij}
		\\	&+g\left(p_i+q_i\right)\frac{ w _j\left(\sqrt{s}\left(p^0+q^0+\sqrt{s}\right)\right)-\left(p+q\right)\cdot w \frac{\partial}{\partial p_j}\left(\sqrt{s}\left(p^0+q^0+\sqrt{s}\right)\right)}{s\left(p^0+q^0+\sqrt{s}\right)^2}
		\bigg)\\	&=
			\left(\frac{1}{2}+\frac{1}{2}g\frac{\left(p+q\right)\cdot w }{\sqrt{s}\left(p^0+q^0+\sqrt{s}\right)}\right)\delta_{ij}
			\\
			&\quad+
			\frac{1}{2g}\left(\frac{q^0}{p^0}p_j-q_j\right)\left( w _i+\left(\gamma-1\right)\left(p_i+q_i\right)\frac{\left(p+q\right)\cdot w }{|p+q|^2}\right)
			\\
		&\quad	+
			\frac{1}{2}g\left(p_i+q_i\right)\left(\frac{ w _j}{\sqrt{s}\left(p^0+q^0+\sqrt{s}\right)}\right)
			\\
			&\quad -
			\frac{\left(p+q\right)\cdot w }{2s\left(p^0+q^0+\sqrt{s}\right)^2}g\left(p_i+q_i\right)\left(\frac{p^0+q^0+2\sqrt{s}}{\sqrt{s}}\left(\frac{q^0}{p^0}p_j-q_j\right)+\frac{\sqrt{s}}{p^0}p_j\right).
			\end{align*}
			Therefore, the terms that contain $p_ip_j$ in the representation above are 
			$$ \frac{1}{2g}\left(\frac{q^0}{p^0}p_j\right)\left(\left(\gamma-1\right)\left(p_i\right)\frac{\left(p+q\right)\cdot w }{|p+q|^2}\right)$$and 
			$$-\frac{\left(p+q\right)\cdot w }{2s\left(p^0+q^0+\sqrt{s}\right)^2}g\left(p_i\right)\left(\frac{p^0+q^0+2\sqrt{s}}{\sqrt{s}}\left(\frac{q^0}{p^0}p_j\right)+\frac{\sqrt{s}}{p^0}p_j\right).$$
			Therefore, the sum of them are equal to 
			\begin{multline*}\frac{1}{2g}\left(\frac{q^0}{p^0}p_j\right)\left(\left(\gamma-1\right)\left(p_i\right)\frac{\left(p+q\right)\cdot w }{|p+q|^2}\right)\\-\frac{\left(p+q\right)\cdot w }{2s\left(p^0+q^0+\sqrt{s}\right)^2}g\left(p_i\right)\left(\frac{p^0+q^0+2\sqrt{s}}{\sqrt{s}}\left(\frac{q^0}{p^0}p_j\right)+\frac{\sqrt{s}}{p^0}p_j\right)\\
			=\frac{1}{2g}\left(\frac{q^0}{p^0}p_j\right)\left(\left(p_i\right)\frac{\left(p+q\right)\cdot w }{\sqrt{s}(p^0+q^0+\sqrt{s})}\right)\\-\frac{\left(p+q\right)\cdot w }{2s\left(p^0+q^0+\sqrt{s}\right)^2}g\left(p_i\right)\left(\frac{p^0+q^0+2\sqrt{s}}{\sqrt{s}}\left(\frac{q^0}{p^0}p_j\right)+\frac{\sqrt{s}}{p^0}p_j\right)\\
			=\frac{p_ip_j\left(p+q\right)\cdot w }{2gp^0\left(p^0+q^0+\sqrt{s}\right)^2s^{\frac{3}{2}}}\bigg(q^0s(p^0+q^0+\sqrt{s})-g^2(p^0+q^0+2\sqrt{s})q^0-sg^2\bigg).
			\end{multline*}
		On the other hand, the terms that contain $q_iq_j$ are 
		$$ \frac{1}{2g}\left(-q_j\right)\left(\left(\gamma-1\right)\left(q_i\right)\frac{\left(p+q\right)\cdot w }{|p+q|^2}\right)$$and 
			$$-\frac{\left(p+q\right)\cdot w }{2s\left(p^0+q^0+\sqrt{s}\right)^2}g\left(q_i\right)\left(\frac{p^0+q^0+2\sqrt{s}}{\sqrt{s}}\left(-q_j\right)\right).$$
			Therefore, the sum of them are equal to 
			\begin{multline*}\frac{1}{2g}\left(-q_j\right)\left(\left(\gamma-1\right)\left(q_i\right)\frac{\left(p+q\right)\cdot w }{|p+q|^2}\right)\\-\frac{\left(p+q\right)\cdot w }{2s\left(p^0+q^0+\sqrt{s}\right)^2}g\left(q_i\right)\left(\frac{p^0+q^0+2\sqrt{s}}{\sqrt{s}}\left(-q_j\right)\right)\\
			=\frac{1}{2g}\left(-q_j\right)\left(\left(q_i\right)\frac{\left(p+q\right)\cdot w }{\sqrt{s}(p^0+q^0+\sqrt{s})}\right)\\-\frac{\left(p+q\right)\cdot w }{2s\left(p^0+q^0+\sqrt{s}\right)^2}g\left(q_i\right)\left(\frac{p^0+q^0+2\sqrt{s}}{\sqrt{s}}\left(-q_j\right)\right)\\
			=\frac{q_iq_j\left(p+q\right)\cdot w }{2g\left(p^0+q^0+\sqrt{s}\right)^2s^{\frac{3}{2}}}\bigg(-s(p^0+q^0+\sqrt{s})+g^2(p^0+q^0+2\sqrt{s})\bigg).
			\end{multline*}
		In addition, note that the coefficients of $p_ip_j$ and $q_ip_j$ are equal. Similarly, the coefficients of $p_iq_j$ and $q_iq_j$ are equal.  
			
			Putting all these together, we can write:
			$$
			\frac{\partial u_i}{\partial p_j}=A\delta_{ij}+Bp_ip_j+Cq_iq_j+Dp_iq_j+Eq_ip_j+Fp_i w _j+Gq_i w _j+H w _ip_j+I w _iq_j,
			$$
			where the scalars are \eqref{def.A} and
			\begin{align}
			B&=\theta\frac{\left(p+q\right)\cdot w }{2gp^0\left(p^0+q^0+\sqrt{s}\right)^2s^{\frac{3}{2}}}\left(q^0s\left(p^0+q^0+\sqrt{s}\right)-g^2q^0\left(p^0+q^0+2\sqrt{s}\right)-g^2s\right), \notag
			\\
			C&=\theta\frac{\left(p+q\right)\cdot w }{2g\left(p^0+q^0+\sqrt{s}\right)^2s^{\frac{3}{2}}}\left(-s\left(p^0+q^0+\sqrt{s}\right)+g^2\left(p^0+q^0+2\sqrt{s}\right)\right), \notag
			\\
			D&=C, \notag
			\\  E&=B, \notag \\
			F&=\theta\frac{g}{2\sqrt{s}\left(p^0+q^0+\sqrt{s}\right)}, \notag \\  G&=F, \notag \\
			H&=\theta\frac{q^0}{2gp^0}, \label{scalars.H}
			\\  I&=-\frac{\theta}{2g}. \notag
			\end{align}
			We will use these notations above throughout the proof.

Notice from \eqref{gamma.calc} and \eqref{def.A} that $\left(1-\theta\right)<A<1$ since
\begin{equation}\label{ineq.A}
\left|g\frac{\left(\gamma-1\right)\left(p+q\right)\cdot w }{|p+q|^2}\right|
=
\left|\frac{g\left(p+q\right)\cdot w }{\sqrt{s}\left(p^0+q^0+\sqrt{s}\right)}\right|
\le \frac{g}{\sqrt{s}}\frac{|p+q|}{p^0+q^0}<1,
\end{equation}
as $s=g^2+4>g^2$ and $|p+q|<p^0+q^0$.
	We will use this to compute the determinant of the matrix $\mathbf{\Phi}=\left(\Phi_{ij}\right)$ where $\Phi_{ij}=\frac{\partial u_i}{\partial p_j}$. 

	We first decompose the pre-collisional vector $p$ as below:
	$$
	p=\left(p\cdot  w \right) w  + w  \times\left(p\times w \right).
	$$
	Define $\bar{ w }\eqdef \frac{ w  \times\left(p\times w \right)}{|p\times w |}$.
	Then, $\bar{ w }\in \mathbb{S}^2$ and $\bar{ w }\perp w $.  Also, define $\tilde{ w }\eqdef \frac{p\times  w }{|p\times  w |}$. 
	
	Then, $\tilde{ w }\in \mathbb{S}^2$ and $\tilde{ w }\perp w $ and $\tilde{ w }\perp\bar{ w }$.
	Thus, $\{ w ,\bar{ w },\tilde{ w }\}$ is an orthonormal basis for $\mathbb{R}^3$. Then, we can decompose $q$ as below:
	\begin{align*}
	q&=\left(q\cdot w \right) w + w \times\left(q\times  w \right)
	\\
	&=\left(q\cdot w \right) w +\left(\left( w \times\left(q\times  w \right)\right)\cdot\bar{ w }\right)\bar{ w }+\bar{ w }\times\left(\left( w \times\left(q\times  w \right)\right)\times\bar{ w }\right)
	\\
	&=\left(q\cdot  w \right) w +\left(q\cdot \bar{ w }\right)\bar{ w }+\left(q\cdot\tilde{ w }\right)\tilde{ w }
	\\
	&\eqdef a w  +b\bar{ w }+c\tilde{ w }.
	\end{align*}
	Here we record
	\begin{equation}\label{def.abc}
		a=\left(q\cdot w \right), \quad b = \left(q\cdot \bar{ w }\right), \quad c = \left(q\cdot\tilde{ w }\right).
	\end{equation}
	Similarly, write:
	\begin{align*}
	p=\left(p\cdot  w \right) w  + w  \times\left(p\times w \right)=\left(p\cdot w \right) w +|p\times w |\bar{ w }\eqdef d w +e\bar{ w }.
	\end{align*}
	Here we also record
	\begin{equation}\label{def.de}
		d=\left(p\cdot w \right), \quad e = |p\times w |.
	\end{equation}
	Notice that $a^2+b^2+c^2=|q|^2$ and $d^2+e^2=|p|^2$.

Then, we can rewrite the matrix element $\Phi_{ij}$:
		\begin{align*}
			\Phi_{ij}&=A\delta_{ij}+B' w _i w _j+C'\bar{ w }_i\bar{ w }_j+D'\tilde{ w }_i\tilde{ w }_j+E' w _i\bar{ w }_j+F' w _i\tilde{ w }_j\\
			&\hspace{10mm}+G'\bar{ w }_i w _j+H'\bar{ w }_i\tilde{ w }_j+I'\tilde{ w }_i w _j+J'\tilde{ w }_i\bar{ w }_j,
		\end{align*}
		where
		\begin{align*}
		B'&=Bd^2+Ca^2+Dad+Ead+Fd+Ga+Hd+Ia
		\\
		C'&=Be^2+Cb^2+Deb+Eeb,
		\\
		D'&=Cc^2
		\\
		E'&=Bde+Cab+Dbd+Eae+He+Ib,
		\\
		F'&=Cac+Dcd+Ic
		\\
		G'&=Bde+Cba+Dae+Ebd+Fe+Gb,
		\\
		H'&=Cbc+Dce
		\\
		I'&=Cac+Ecd+Gc,
		\\ 
		J'&=Cbc+Ece.
		\end{align*}
Therefore, we have
	\begin{multline}\label{matrixPhi}\mathbf{\Phi} = AI+B' w w^\top+C'\bar{ w } \bar{ w }^\top+D'\tilde{ w } \tilde{ w }^\top+E' w  \bar{ w }^\top+F' w  \tilde{ w }^\top\\
			\hspace{10mm}+G'\bar{ w }  w ^\top+H'\bar{ w } \tilde{ w }^\top+I'\tilde{ w }  w ^\top+J'\tilde{ w } \bar{ w }^\top.\end{multline} Note that $\{ w ,\bar{ w },\tilde{ w }\}$ forms an orthonormal basis for $\mathbb{R}^3$. Now we define a 3 by 3 matrix $M$ of the orthonormal basis as 
	$$M=\left( \begin{array}{ccc}
		w& \bar{w} & \tilde{w}\end{array} \right)=\left( \begin{array}{ccc}
		w_1& \bar{w}_1 & \tilde{w}_1 \\
		w_2& \bar{w}_2 & \tilde{w}_2\\
		w_3& \bar{w}_3 & \tilde{w}_3\end{array} \right).$$
		Then the determinant of $\mathbf{\Phi} $ is the same as that of $M^\top \mathbf{\Phi}  M$ as $M$ is the matrix of an orthonormal basis. Now we observe that
		$$(M^\top \mathbf{\Phi}  M)_{11}= w^\top \mathbf{\Phi}  w = A+B',
		$$
		$$(M^\top \mathbf{\Phi}  M)_{21}= \bar{w}^\top \mathbf{\Phi}  w = G',
		$$
		$$(M^\top \mathbf{\Phi}  M)_{31}= \tilde{w}^\top \mathbf{\Phi}  w = I',
		$$
		$$(M^\top \mathbf{\Phi}  M)_{12}= w^\top \mathbf{\Phi}  \bar{w} = E',
		$$
		$$(M^\top \mathbf{\Phi}  M)_{22}= \bar{w}^\top \mathbf{\Phi}  \bar{w} = A+C',
		$$
		$$(M^\top \mathbf{\Phi}  M)_{32}= \tilde{w}^\top \mathbf{\Phi}  \bar{w} = J',
		$$
		$$(M^\top \mathbf{\Phi}  M)_{13}= w^\top \mathbf{\Phi}  \tilde{w} = F',
		$$
		$$(M^\top \mathbf{\Phi}  M)_{23}= \bar{w}^\top \mathbf{\Phi}  \tilde{w} = H',
		$$and 
		$$(M^\top \mathbf{\Phi}  M)_{33}= \tilde{w}^\top \mathbf{\Phi}  \tilde{w} = A+D'.
		$$
Therefore, the determinant of $\mathbf{\Phi}$ is equal to:
		$$
		\det\left(\mathbf{\Phi}\right)=\det\left(M^\top \mathbf{\Phi}  M\right)=
		 \left| \begin{array}{ccc}
		A+B' & E' & F' \\
		G' & A+C' & H' \\
		I' & J' & A+D' \end{array} \right|.
		$$
We will further row reduce this determinant to obtain the expression in \eqref{det.up}.  We write the matrix $\mathbf{\Phi} = \left(\mathbf{\Phi}_{ij}\right)$ to represent the components.

	Subtracting (Column 3)$\times\frac{a}{c}$ from (Column 1) and subtracting (Column 3)$\times\frac{b}{c}$ from (Column 2) gives 
	\begin{align*}
	\Phi_{11}&=A+Bd^2+Ead+Fd+Ga+Hd,
	\\
	\Phi_{21}&=Bde+Ebd+Fe+Gb,
	\\ 
	\Phi_{31}&=Ecd+Gc-\frac{a}{c}A
	\\
	\Phi_{12}&=Bde+Eae+He,
	\\
	\Phi_{22}&=A+Be^2+Ebe,
	\\
	\Phi_{32}&=Ece-\frac{b}{c}A.
	\end{align*}
	There is no change on Column 3 by this column reduction. These row reductions do not change the determinant. 

	Now, subtracting (Column 2)$\times \frac{d}{e}$ from (Column 1) gives
	\begin{gather*}
	\Phi_{11}=A+Fd+Ga,\hspace{2mm}
	\Phi_{21}=-\frac{d}{e}A+Fe+Gb,\hspace{2mm}
	\Phi_{31}=\left(\frac{bd}{ce}-\frac{a}{c}\right)A+Gc.
	\end{gather*}
	Now, we subtract (Row 3)$\times \frac{a}{c}$ from (Row 1) and (Row 3)$\times \frac{b}{c}$ from (Row 2) respectively. Then, we have the matrix elements to be:
	\begin{align*}
	\Phi_{11}&=\left(1-\frac{abd}{c^2e}+\frac{a^2}{c^2}\right)A+Fd,
	\\ 
	\Phi_{21}&=\left(\frac{ab}{c^2}-\frac{d}{e}-\frac{b^2d}{c^2e}\right)A+Fe
	\\
	\Phi_{31}&=\left(\frac{bd}{ce}-\frac{a}{c}\right)A+Gc,
	\\
	\Phi_{12}&=\frac{ab}{c^2}A+Bde+He,
	\\
	\Phi_{22}&=\left(1+\frac{b^2}{c^2}\right)A+Be^2
	\\
	\Phi_{32}&=-\frac{b}{c}A+Ece,
	\\
	\Phi_{13}&=-\frac{a}{c}A+Dcd+Ic,
	\\
	\Phi_{23}&=-\frac{b}{c}A+Dce,
	\\ 
	\Phi_{33}&=A+Cc^2.
	\end{align*}
	We do one more row reduction: subtract (Row 2)$\times\frac{d}{e}$ from (Row 1). This gives
$$
\det\left(\mathbf{\Phi}\right)=
	\left| \begin{array}{ccc}
	\aoo A & \aot A+He & \aox A+Ic \\
	\ato A+Fe & \att A+Be^2 & \atx A+Dce \\
	\axo A+Gc & \axt A+Ece & \axx A+Cc^2 \end{array} \right|
$$
where 
\begin{align*}
	\aoo&=1-\frac{abd}{c^2e}+\frac{a^2}{c^2}-\frac{abd}{c^2e}+\frac{d^2}{e^2}+\frac{b^2d^2}{c^2e^2}
	\\
	\ato&=\frac{ab}{c^2}-\frac{d}{e}-\frac{b^2d}{c^2e},
	\\ 
	\axo&=\frac{bd}{ce}-\frac{a}{c},
	\\ 
	\aot&=\frac{ab}{c^2}-\frac{d}{e}-\frac{b^2d}{c^2e}
	\\
	\att&=1+\frac{b^2}{c^2},
	\\ 
	\axt&=-\frac{b}{c},
	\\ 
	\aox&=\frac{bd}{ce}-\frac{a}{c},
	\\ 
	\atx&=-\frac{b}{c},
	\\
	\axx&=1.
\end{align*}
	Since $B=E$, $C=D$, and $G=F$, we can do one more row reduction: (Row 2)-(Row 3)$\times \frac{e}{c}$. This gives
	\begin{align*}
	\Phi_{21}&=\left(\frac{ab}{c^2}-\frac{d}{e}-\frac{b^2d}{c^2e}-\frac{bd}{c^2}+\frac{ae}{c^2}\right)A\eqdef \ato'A
	\\
	\Phi_{22}&=\left(1+\frac{b^2}{c^2}+\frac{be}{c^2}\right)A\eqdef \att'A,
	\\ 
	\Phi_{23}&=\left(-\frac{b}{c}-\frac{e}{c}\right)A\eqdef \atx'A.
	\end{align*}  
	Finally, we have 
$$
	\det\left(\mathbf{\Phi}\right)=
	 \left| \begin{array}{ccc}
	\aoo A & \aot A+He & \aox A+Ic \\
	\ato' A & \att' A & \atx' A \\
	\axo A+Gc & \axt A+Ece & \axx A+Cc^2 \end{array} \right|,$$
	where
	\begin{align*}
	\aoo&=\frac{c^2|p|^2+L^2}{c^2e^2},
	\\ 
	\ato'&=\frac{L\left(b+e\right)-c^2d}{c^2e},
	\\ 
	\axo&=-\frac{L}{ce},
	\\
	\aot&=\frac{bL-c^2d}{c^2e},
	\\ 
	\att'&=\frac{b^2+c^2+be}{c^2},
	\\ 
	\axt&=-\frac{b}{c},
	\\
	\aox&=-\frac{L}{ce},
	\\ 
	\atx'&=-\frac{b+e}{c},
	\\ 
	\axx&=1,
	\end{align*}
	with $L\eqdef ae-bd$.

	Then the determinant is
	\begin{align*}
	\det\left(\mathbf{\Phi}\right)=A&\bigg(\aoo A\left(\att'A+Cc^2\att'-\atx'\axt A-\atx' Ece\right)\\
	&-\left(\aot A+He\right)\left(\ato' A+\ato' Cc^2-\atx'\axo A-\atx'Gc\right)\\
	&+\left(\aox A+Ic\right)\left(\ato'\axt A+\ato' Ece-\att'\axo A-\att'Gc\right)\bigg).
	\end{align*}
	Here we further reduce the determinant. First, notice that
	$$
	\att'=\frac{b^2+c^2+be}{c^2}=1+\left(-\frac{b}{c}\right)\left(-\frac{b+e}{c}\right)=1+\axt\atx'.
	$$
Thus, we obtain
$$
	\att' A-\atx'\axt A=A.
	$$
	Also, we have
$$
	\ato'-\atx'\axo=\frac{L\left(b+e\right)-c^2d}{c^2e}-\left(\frac{-b-e}{c}\right)\left(\frac{-L}{ce}\right)=-\frac{d}{e},
$$
	and
$$
	\ato'\axt-\att'\axo=\frac{L\left(b+e\right)-c^2d}{c^2e}\left(\frac{-b}{c}\right)-\left(\frac{b^2+be}{c^2}+1\right)\left(\frac{-L}{ce}\right)=\frac{a}{c}.
$$
	Then the determinant is now
	\begin{align*}
	\det\left(\mathbf{\Phi}\right)&=A\bigg(\aoo A\left(A+Cc^2\att'-\atx' Ece\right)\\
	&\hspace*{5mm}-\left(\aot A+He\right)\left(-\frac{d}{e}A+\ato' Cc^2-\atx'Gc\right)\\
	&\hspace*{5mm}+\left(\aox A+Ic\right)\left(\frac{a}{c}A+\ato' Ece-\att'Gc\right)\bigg)\\
	&=\left(\aoo+\aot\frac{d}{e}+\aox\frac{a}{c}\right)A^3\\
	&\hspace*{5mm}+\bigg(\aoo\att'Cc^2-\aoo\atx'Ece+Hd-\aot\ato'Cc^2\\
	&\hspace*{5mm}+\aot\atx'Gc+Ia+\aox\ato'Ece-\aox\att'Gc\bigg)A^2\\
	&\hspace*{5mm}+\left(-\ato'CHc^2e+\atx'GHce+\ato'IEc^2e-\att'IGc^2\right)A.
	\end{align*}
	Thus 
	$$
	\det\left(\mathbf{\Phi}\right) = P_1A^3+P_2A^2+P_3A,
	$$
	where 
\begin{align}
		P_1&=\aoo+\aot\frac{d}{e}+\aox\frac{a}{c}
		\label{def.p1}
		\\
		P_2&=\aoo\att'Cc^2-\aoo\atx'Ece+Hd-\aot\ato'Cc^2
		\label{def.p2}
		\\
		&
		\hspace*{5mm}
		+\aot\atx'Gc+Ia+\aox\ato'Ece-\aox\att'Gc,
		\notag
		\\
		P_3&=-\ato'CHc^2e+\atx'GHce+\ato'IEc^2e-\att'IGc^2.
		\label{def.p3}
\end{align}
We will further simplify and compute $P_1$, $P_2$ and $P_3$ from \eqref{def.p1}, \eqref{def.p2} and \eqref{def.p3} below.

	We compute $P_1$ from \eqref{def.p1} first. It can be simply estimated as
	\begin{align*}
	P_1&=\frac{c^2|p|^2+L^2}{c^2e^2}+\frac{bdL-c^2d^2}{c^2e^2}-\frac{Lae}{c^2e^2}\\
	&=\frac{1}{c^2e^2}\left(c^2e^2+L^2+L\left(bd-ae\right)\right)
	=\frac{1}{c^2e^2}\left(c^2e^2+L^2-L^2\right)=1.
	\end{align*}
	This is all we need for $P_1$.
	
	We now simplify $P_2$ from \eqref{def.p2}. From the previous calculations we have
	\begin{align*}
	P_2=&Cc^2\left(\aoo\att'-\aot\ato'\right)+Ece\left(\aox\ato'-\aoo\atx'\right)\\
	&+Gc\left(\aot\atx'-\aox\att'\right)+\theta\pqhat.
	\end{align*}
	We first have
	\begin{align*}
	&\aoo\att'-\aot\ato'\\
	&=\frac{1}{c^4e^2}\left(\left(c^2|p|^2+L^2\right)\left(b^2+c^2+be\right)-\left(bL-c^2d\right)\left(L\left(e+b\right)-c^2d\right)\right)\\
	&=\frac{1}{c^4e^2}\left(c^4\left(|p|^2-d^2\right)+b^2c^2|p|^2+bc^2\left(|p|^2e+2Ld\right)+c^2\left(L^2+Lde\right)\right)\\
	&=1+\frac{1}{c^2e^2}\left(b^2|p|^2+b\left(|p|^2e+2Ld\right)+L^2+Lde\right)\\
	&=1+\frac{1}{c^2e^2}\left(b^2e^2+|p|^2be+e^2a^2+e^2ad-bd^2e\right)\\
	&=\frac{1}{c^2e^2}\left(|q|^2e^2+be^3+e^2ad\right)
	=\frac{1}{c^2}\left(|q|^2+be+ad\right).
	\end{align*}
	We also have
	\begin{align*}
	&\aox\ato'-\aoo\atx'
	=\frac{1}{c^3e^2}\left(-L^2\left(e+b\right)+Lc^2d+\left(c^2|p|^2+L^2\right)\left(b+e\right)\right)\\
	&=\frac{1}{c^3e^2}\left(Lc^2d+bc^2|p|^2+c^2|p|^2e\right)\\
	&=\frac{1}{c^3e^2}\left(c^2dea-bc^2d^2+bc^2|p|^2+c^2|p|^2e\right)=\frac{1}{ce}\left(ad+be+|p|^2\right).
	\end{align*}
	Lastly, we observe
	\begin{align*}
	\aot\atx'-\aox\att'
	&=\frac{1}{c^3e}\left(\left(bL-c^2d\right)\left(-b-e\right)+L\left(b^2+c^2+be\right)\right)\\
	&=\frac{1}{c^3e}\left(c^2d\left(b+e\right)+Lc^2\right)\\
	&=\frac{1}{c^3e}\left(c^2db+c^2de+c^2ea-c^2bd\right)=\frac{a+d}{c}.
	\end{align*}
	Thus, we put these computations together to obtain
	\begin{multline}\label{P2representation}
	P_2=C\left(ad+be+|q|^2\right)+E\left(ad+be+|p|^2\right)+G\left(a+d\right)\\+\theta\pqhat.
	\end{multline}
	We also have using $s=g^2+4$ and \eqref{gamma.calc} that
	\begin{align*}
	E=B&=\frac{\theta}{p^0}\CE\left(4q^0\left(\pqs\right)-g^2\sqs\left(q^0+\sqs\right)\right)\\
	&=\theta\frac{\left(a+d\right)}{2gp^0}\frac{\left(\gamma-1\right)^2}{|p+q|^4}\left(4q^0\left(\gamma+1\right)-g^2\left(q^0+\sqs\right)\right).
	\end{align*}
	We again use \eqref{gamma.calc} to obtain that
	\begin{align*}
	C=D&=\theta\CE\left(-4\left(\pqs\right)+g^2\sqs\right)\\
	&=\theta\frac{\left(a+d\right)}{2g}\frac{\left(\gamma-1\right)^2}{|p+q|^4}\left(-4\left(\gamma+1\right)+g^2\right),
	\end{align*}
	and further using \eqref{gamma.calc} again
$$
G=F=\theta\frac{g}{2\sqs\left(\pqs\right)}=\theta\frac{\left(\gamma-1\right)g}{2|p+q|^2}.
	$$
Then we reduce $P_2$ as
	\begin{align*}
	P_2=&C\left(ad+be+|q|^2\right)+E\left(ad+be+|p|^2\right)+G\left(a+d\right)+\theta\pqhat\\
	=&\theta\frac{\left(\gamma-1\right)^2\left(p+q\right)\cdot w }{2gp^0|p+q|^4}\bigg(\left(|q|^2+ad+be\right)\left(p^0g^2-4p^0\left(\gamma+1\right)\right)\\
	&\hspace*{20mm}+\left(|p|^2+ad+be\right)\left(4q^0\left(\gamma+1\right)-g^2\left(q^0+\sqs\right)\right)\\
	&\hspace*{20mm}+\frac{|p+q|^2}{\left(\gamma-1\right)}p^0g^2+\frac{|p+q|^4}{\left(\gamma-1\right)^2}\frac{\left(q^0d-p^0a\right)}{a+d}\bigg)\\
	=&\theta\frac{\left(\gamma-1\right)^2\left(p+q\right)\cdot w }{2gp^0|p+q|^4}\bigg(\left(ad+be\right)\left(4\left(q^0-p^0\right)\left(\gamma+1\right)+g^2\left(p^0-q^0\right)-g^2\sqs\right)\\
	&\hspace*{20mm}+\left(4\gamma+4-g^2\right)\left(|p|^2q^0-|q|^2p^0\right)-g^2|p|^2\sqs\\
	&\hspace*{20mm}+\frac{|p+q|^2}{\left(\gamma-1\right)}p^0g^2+\frac{|p+q|^4}{\left(\gamma-1\right)^2}\frac{\left(q^0d-p^0a\right)}{a+d}\bigg)\\
	=&\theta\frac{\left(\gamma-1\right)^2\left(p+q\right)\cdot w }{2gp^0|p+q|^4}\bigg(\left(ad+be\right)\left(\left(q^0-p^0\right)\left(4\gamma+4-g^2\right)-g^2\sqs\right)\\
	&\hspace*{20mm}+\left(4\gamma+4-g^2\right)\left(p^0-q^0\right)\left(1+p^0q^0\right)-g^2|p|^2\sqs\\
	&\hspace*{20mm}+\frac{|p+q|^2}{\left(\gamma-1\right)}p^0g^2+\frac{|p+q|^4}{\left(\gamma-1\right)^2}\frac{\left(q^0d-p^0a\right)}{a+d}\bigg).
	\end{align*}
	We note that $a+d=\left(p+q\right)\cdot w$.
	Thus, we obtain from the above that 
	\begin{equation}
	\label{P2}
\begin{split}
P_2	=&
\theta \frac{\left(\gamma-1\right)^2\left(p+q\right)\cdot w }{2gp^0|p+q|^4}
\left(1+p^0q^0-ad-be\right)\left(p^0-q^0\right)\left(4\gamma+4-g^2\right)
\\
& 
-
\theta \frac{\left(\gamma-1\right)^2\left(p+q\right)\cdot w }{2gp^0|p+q|^4}
\left(ad+be+|p|^2\right)g^2\sqs
\\
&
+
\theta \frac{\left(\gamma-1\right)\left(p+q\right)\cdot w }{2gp^0|p+q|^2}
p^0g^2
+
\theta \frac{\left(p+q\right)\cdot w }{2gp^0}
\frac{\left(q^0d-p^0a\right)}{a+d}.
\end{split}
	\end{equation}
	For an upper-bound estimate for $|P_2|$, we estimate each term in \eqref{P2}. First of all, we note from \eqref{def.abc} and \eqref{def.de} that
	\begin{equation}\label{abcde.bounds}
		|a|,|b|,|c| \lesssim q^0, \quad |d|,|e|\lesssim p^0.
	\end{equation}
	We also have from for instance \eqref{g2} that 
	\begin{equation}\label{gs2.bounds}
		g\leq\sqrt{s}\leq 2\sqrt{p^0q^0}.
	\end{equation}
	Finally, we observe that 
	$$\left|\frac{p^0-q^0}{g}\right|\leq \frac{|p-q|}{g}\leq \sqrt{p^0q^0},$$ 
	which holds by \eqref{gINEQ}. Therefore, we obtain that the first term in \eqref{P2} is bounded above as
	\begin{multline*}
	\theta\left| \frac{\left(\gamma-1\right)^2\left(p+q\right)\cdot w }{2gp^0|p+q|^4}\left(1+p^0q^0-ad-be\right)\left(p^0-q^0\right)\left(4\gamma+4-g^2\right)\right|\\
	=\theta\left| \frac{\left(p+q\right)\cdot w }{2gp^0s(p^0+q^0+\sqrt{s})^2}\left(1+p^0q^0-ad-be\right)\left(p^0-q^0\right)\left(4\gamma+4-g^2\right)\right|\\
	\leq \theta \frac{\sqrt{p^0q^0}\left|p+q\right| }{2p^0s(p^0+q^0+\sqrt{s})^2}\left|1+p^0q^0-ad-be\right|\left|4\gamma+4-g^2\right|\\
	\lesssim   \frac{\sqrt{p^0q^0}\left|p+q\right| }{2p^0s(p^0+q^0+\sqrt{s})^2}(p^0q^0)\max\{p^0+q^0, s\}\lesssim \frac{q^0\sqrt{p^0q^0}}{p^0+q^0}\frac{\max\{p^0+q^0, s\}}{s},\end{multline*}
	where we again used \eqref{gamma.calc} and that 
$$|4\gamma+4-g^2|\le 4\gamma+4+g^2= \frac{4(p^0+q^0)}{\sqrt{s}}+s \le 4 \max\{p^0+q^0,s\}.$$
	\begin{enumerate}
	\item In the case that $\max\{p^0+q^0, s\}=s,$ if $p^0\geq q^0$, we observe that $$\frac{q^0\sqrt{p^0q^0}}{p^0+q^0}\frac{\max\{p^0+q^0, s\}}{s}=\frac{q^0\sqrt{p^0q^0}}{p^0+q^0}\leq \frac{q^0p^0}{p^0+q^0}\leq  q^0.$$ On the other hand, if $q^0\geq p^0$, we observe that $$\frac{q^0\sqrt{p^0q^0}}{p^0+q^0}\frac{\max\{p^0+q^0, s\}}{s}=\frac{q^0\sqrt{p^0q^0}}{p^0+q^0}\leq q^0.$$Therefore, we conclude
\begin{equation}
\label{p21estimate}
	\theta\left| \frac{\left(\gamma-1\right)^2\left(p+q\right)\cdot w }{2gp^0|p+q|^4}\left(1+p^0q^0-ad-be\right)\left(p^0-q^0\right)\left(4\gamma+4-g^2\right)\right|\lesssim q^0.
\end{equation}
\item In the case that $\max\{p^0+q^0, s\}=p^0+q^0,$ we observe that
$$\frac{q^0\sqrt{p^0q^0}}{p^0+q^0}\frac{\max\{p^0+q^0, s\}}{s}= \frac{q^0\sqrt{p^0q^0}}{s}.$$Therefore, we conclude
\begin{multline}
\label{p21estimate2}
	\theta\left| \frac{\left(\gamma-1\right)^2\left(p+q\right)\cdot w }{2gp^0|p+q|^4}\left(1+p^0q^0-ad-be\right)\left(p^0-q^0\right)\left(4\gamma+4-g^2\right)\right|\\ \lesssim \frac{(p^0)^{\frac{1}{2}}(q^0)^{\frac{3}{2}}}{s}.
\end{multline}
\end{enumerate}
We now estimate the second term in the (RHS) of \eqref{P2}. We observe that the second term is bounded above as
\begin{multline}
\label{p22estimate}
	\theta\left| \frac{\left(\gamma-1\right)^2\left(p+q\right)\cdot w }{2gp^0|p+q|^4}\left(ad+be+|p|^2\right)g^2\sqs\right|
	\\
	=\theta\left| \frac{\left(p+q\right)\cdot w }{2gp^0s(p^0+q^0+\sqrt{s})^2}\left(ad+be+|p|^2\right)g^2\sqs\right|\\
	\lesssim \frac{|p+q|(p^0q^0+|p|^2)g}{p^0\sqrt{s}(p^0+q^0)^2}\lesssim 1.
	\end{multline}
We now estimate the third term in the (RHS) of \eqref{P2}. We observe that the third term is bounded above as
\begin{equation}
\label{p23estimate}
	\theta\left| \frac{\left(\gamma-1\right)\left(p+q\right)\cdot w }{2|p+q|^2}g\right|=	\theta\left| \frac{\left(p+q\right)\cdot w }{2\sqrt{s}(p^0+q^0+\sqrt{s})}g\right|\leq \theta.
	\end{equation}
	Finally, using \eqref{def.abc} and \eqref{def.de}, we estimate the last term in the (RHS) of \eqref{P2} as below:
	\begin{multline}
\label{p24estimate}
	\theta\left| \frac{\left(p+q\right)\cdot w }{2gp^0}\frac{q^0d-p^0a}{a+d}\right|=	\theta\left| \frac{ (q^0p-p^0q)\cdot w}{2gp^0}\right|\\  \leq 	\theta\left| \frac{ (q^0p-q^0q)\cdot w}{2gp^0}\right|+\theta\left| \frac{ (q^0q-p^0q)\cdot w}{2gp^0}\right|\\
	\leq \theta\frac{q^0 |p-q|}{2gp^0}+\theta\frac{|q||p^0-q^0|}{2gp^0}\leq \theta \frac{q^0|p-q|}{gp^0}\leq \theta \frac{q^0\sqrt{p^0q^0}}{p^0}=\theta \frac{(q^0)^{3/2}}{(p^0)^{1/2}},
	\end{multline}
	where we used $|p^0-q^0|\leq |p-q|$ and \eqref{gINEQ}.
	Together with \eqref{p21estimate},  \eqref{p21estimate2}, \eqref{p22estimate}, and \eqref{p23estimate}, we have that 
	\begin{equation}
	\label{P2finalestimate}
	|P_2|\lesssim q^0\left(1+\frac{\sqrt{p^0q^0}}{s}\right) + \frac{(q^0)^{\frac{3}{2}}}{(p^0)^{\frac{1}{2}}}\lesssim (q^0)^{\frac{3}{2}}\left(1+\frac{\sqrt{p^0}}{s}\right).
	\end{equation}
	This completes our estimates for $P_2$.
	
	We now simplify $P_3$. Recall from \eqref{def.p3} that
	\begin{equation}\label{P3representation}
	\begin{split}
	P_3&=-\ato'CHc^2e+\atx'GHce+\ato'IEc^2e-\att'IGc^2\\
	&=\ato'c^2e\left(IE-CH\right)+G\left(\atx'Hce-\att'Ic^2\right)\\
	&=\left(abe-bde+ae^2-c^2d-b^2d\right)\left(IE-CH\right)\\&\qquad\qquad\qquad\qquad-G\left(\left(b+e\right)He+I\left(b^2+c^2+be\right)\right).
	\end{split}\end{equation}
	From \eqref{scalars.H} and  \eqref{gamma.calc} we have that
	\begin{align*}
	IE-CH&=-\frac{\theta}{2g}E-\theta\frac{q^0}{2gp^0}C\\
	&=-\frac{\theta}{2g}\left(-C\frac{q^0}{p^0}-\frac{\left(a+d\right)g\sqs}{2p^0}\frac{\left(\gamma-1\right)^2}{|p+q|^4}+C\frac{q^0}{p^0}\right)=\theta^2\frac{\left(a+d\right)\left(\gamma-1\right)^2\sqs}{4p^0|p+q|^4}.
	\end{align*}
Notice from \eqref{scalars.H} that the second term in \eqref{P3representation} is equal to
\begin{multline*}
    -G\left(\left(b+e\right)He+I\left(b^2+c^2+be\right)\right)\\=-\theta\frac{g}{2\sqrt{s}(p^0+q^0+\sqrt{s})}\left(\left(be+e^2\right)\theta \frac{q^0}{2gp^0}-\frac{\theta}{2g}\left(b^2+c^2+be\right)\right)\\
    =\theta^2\frac{1}{4p^0\sqrt{s}(p^0+q^0+\sqrt{s})}\bigg(
	\left(q^0\left(-be-e^2\right)+p^0\left(b^2+c^2+be\right)\right)\bigg)\\
    =\theta^2\frac{\left(\gamma-1\right)}{4p^0|p+q|^2}\bigg(
	\left(q^0\left(-be-e^2\right)+p^0\left(b^2+c^2+be\right)\right)\bigg),
\end{multline*}
using again \eqref{gamma.calc}.
Then we can further reduce $P_3$, using \eqref{gamma.calc}, as
	\begin{align}\label{P3.appendix}
	P_3=&\theta^2\frac{\left(a+d\right)\left(\gamma-1\right)^2\sqs}{4p^0|p+q|^4}\left(abe-bde+ae^2-c^2d-b^2d\right)
	\\
	\notag
	&\hspace*{7mm}+\theta^2\frac{\left(\gamma-1\right)}{4p^0|p+q|^2}\bigg(
	\left(q^0\left(-be-e^2\right)+p^0\left(b^2+c^2+be\right)\right)\bigg)
	\\
	\notag
	=&\theta^2\frac{\left(\gamma-1\right)}{4p^0|p+q|^2}\bigg(\frac{\left(a+d\right)\left(abe-bde+ae^2-c^2d-b^2d\right)}{\pqs}
	\\
	\notag
	&\hspace*{18mm}+be\left(p^0-q^0\right)+p^0\left(|q|^2-a^2\right)-q^0\left(|p|^2-d^2\right)\bigg)
	\\
	\notag
	=&\theta^2\frac{\left(\gamma-1\right)}{4p^0|p+q|^2}\frac{1}{\pqs}\bigg(\left(e^2a-bde+abe-b^2d-c^2d\right)\left(a+d\right)
	\\
	\notag
	&\hspace*{20mm}+\left(\pqs\right)\left(p^0\left(be+|q|^2-a^2\right)+q^0\left(-be-|p|^2+d^2\right)\right)\bigg)
	\\
	\notag
	=&\theta^2\frac{\left(\gamma-1\right)}{4p^0|p+q|^2}\frac{1}{\pqs}\left(I_1+I_2\right).
	\end{align}
	Here, we have
\begin{align*}
I_1&\eqdef \left(e^2a-bde+abe-b^2d-c^2d\right)\left(a+d\right)\\
&=\left(a^2-d^2\right)\left(ad+be\right)+\left(a^2-d^2\right)\left(\left(p^0\right)^2-1\right)+\left(ad+d^2\right)\left(\left(p^0\right)^2-\left(q^0\right)^2\right)\\
&=\left(a^2-d^2\right)\left(ad+be-1\right)+\left(p^0\right)^2\left(a^2+ad\right)-(q^0)^2\left(ad+d^2\right),
\end{align*}
and
$$
I_2\eqdef 
\left(\pqs\right)I_3.
$$
Also
$$
I_3\eqdef p^0\left(be+|q|^2-a^2\right)+q^0\left(-be-|p|^2+d^2\right).
$$
Thus, we have
	\begin{align*}
	I_1+I_2=& \sqs I_3+\left(ad+be\right)\left(a^2-d^2+\left(p^0\right)^2-\left(q^0\right)^2\right)-\left(a^2-d^2\right)\\
	&+\left(p^0\right)^2|q|^2-\left(q^0\right)^2|p|^2+p^0q^0\left(-|p|^2+d^2+|q|^2-a^2\right)\\
	=&\sqs I_3+\left(ad+be-p^0q^0\right)\left(a^2-d^2+\left(p^0\right)^2-\left(q^0\right)^2\right)\\
	&\hspace{10mm}-\left(a^2-d^2\right)+\left(p^0\right)^2|q|^2-\left(q^0\right)^2|p|^2\\
	=&\sqs I_3+\left(1+p^0q^0-ad-be\right)\left(d^2-a^2-\left(p^0\right)^2+\left(q^0\right)^2\right)\\
	=&\sqs I_3+\left(1+p^0q^0-ad-be\right)\left(-e^2+b^2+c^2\right).
	\end{align*}
	Therefore, we finally obtain
\begin{multline}
\label{P3}
	P_3=\theta^2\frac{\left(\gamma-1\right)^2}{4p^0\sqs |p+q|^4}
	\bigg(\sqs\left(\left(be+e^2\right)\left(p^0-q^0\right)\right)
	\\
	+\left(1+p^0q^0-ad-be\right)\left(\sqs p^0-e^2+b^2+c^2\right)\bigg).
\end{multline}
This completes our calculation of $P_3$
	
	We now estimate the upper-bound for $|P_3|$. Recall \eqref{gamma.calc}, 	\eqref{abcde.bounds} and \eqref{gs2.bounds}.
	We further have $s\geq \max\{g^2,4\}.$ Therefore, we have
\begin{multline*}
|P_3|=\theta^2\left|\frac{1}{4p^0s^{3/2}(p^0+q^0+\sqrt{s})^2}\right|\bigg|\sqs\left(\left(be+e^2\right)\left(p^0-q^0\right)\right)\\+\left(1+p^0q^0-ad-be\right)\left(\sqs p^0-e^2+b^2+c^2\right)\bigg|\\
\lesssim \frac{1}{4p^0s^{3/2}(p^0+q^0+\sqrt{s})^2}\bigg|\sqs\left(\left(p^0q^0+(p^0)^2\right)\left(p^0-q^0\right)\right)\\+(p^0q^0)\left(\sqs p^0+|p|^2+|q|^2\right)\bigg|\\
\lesssim  \frac{|p^0-q^0|}{4s(p^0+q^0+\sqrt{s})}+ \frac{q^0((p^0+q^0) p^0+|p|^2+|q|^2)}{4s^{3/2}(p^0+q^0+\sqrt{s})^2}
\lesssim
\frac{1}{s}+\frac{q^0}{s^{3/2}}
\lesssim \frac{q^0}{s}.
\end{multline*}	Therefore,
	 $|P_3|\lesssim \frac{q^0}{s}.$
	This completes the proof.
	\end{proof}
	
	This completes our discussion of the derivation of the Jacobian determinant in \eqref{det.up} and calculating the upper bounds for it in Theorem \ref{Jacobian}.   In the next section we give an example which proves that the Jacobian determinant can become zero.

\section{The lower-bound of the Jacobian of the collision map}\label{sec:lowerBD}
This section is devoted to proving that the the Jacobian determinant det$\left(\frac{\partial u}{\partial p}\right)$ indeed attains the value zero. We prove that in the following theorem:

\begin{theorem}\label{limit.jacobian}
Suppose that the post-collisional momentum $p'$ is defined as \eqref{p'} and $u=\theta p'+(1-\theta)p$ for $\theta \in (0,1)$ as in \eqref{u.split}.  Then we have 
	$$\lim_{\theta \rightarrow 1} \lim_{|q|\rightarrow \infty} \left|\det\left(\frac{\partial u}{\partial p}\right)\right|_{\text{at }p=0 \text{ and } q=-|q|w}=0.$$
\end{theorem} 
\begin{proof}
	We use the formula on the Jacobian determinant from Theorem \ref{Jacobian}. By Theorem \ref{Jacobian}, we have the determinant as \eqref{det.up} with \eqref{def.A}
	and $P_2$ and $P_3$ are defined as in \eqref{P2} and \eqref{P3}, respectively. 
	
	We now compute each value of $A$, $P_2$, and $P_3$ when $p=0$ and $q=-|q|w$. If $p=0$ and $q=-|q|w$, using also \eqref{gamma.calc} we have the following identities:
	\begin{equation}
		\begin{split}
		p^0&=1,\\
		g&=\sqrt{2q^0-2},\\
		s&=2q^0+2,\\
		\gamma&=\frac{1+q^0}{\sqrt{2q^0+2}},\\
		\gamma-1&=\frac{|q|^2}{\sqrt{2q^0+2}(q^0+1+\sqrt{2q^0+2})},\\
		|p+q|&=|q|,\\
		a&=-|q|,\text{ and }\\ b&=c=d=e=0. 
		\end{split}
	\end{equation}
	Then, we can further observe that
	\begin{equation}\label{Alim}A=1-\frac{\theta}{2}-\frac{\theta}{2}\left(\frac{\sqrt{2q^0-2}|q|}{\sqrt{2q^0+2}(q^0+1+\sqrt{2q^0+2})}\right)\rightarrow 1-\theta,\ \text{as}\ |q|\rightarrow \infty. \end{equation} 
We will use this limit at the end of the proof.		
	
Now we study $P_2$, by \eqref{P2}, we have 
	\begin{align}
\notag
	P_2	=&\theta \frac{\left(\gamma-1\right)^2\left(p+q\right)\cdot w }{2gp^0|p+q|^4}\bigg(\left(1+p^0q^0-ad-be\right)\left(p^0-q^0\right)\left(4\gamma+4-g^2\right)\\
	\notag
	&\ -\left(ad+be+|p|^2\right)g^2\sqs\bigg)
	+\theta \frac{\left(\gamma-1\right)g\left(p+q\right)\cdot w }{2|p+q|^2}  +\theta\frac{\left(q^0d-p^0a\right)  }{2gp^0}\\
	\notag
	=&- \frac{\theta|q| }{2\sqrt{2q^0-2}(2q^0+2)(q^0+1+\sqrt{2q^0+2})^2}\\
	\notag 
	&\quad\quad\times\left(\left(1+q^0\right)\left(1-q^0\right)\left(4\frac{1+q^0}{\sqrt{2q^0+2}}-2q^0+6\right)\right)\\
	\notag
	&\ 
	-\theta \frac{\sqrt{2q^0-2}|q| }{2\sqrt{2q^0+2}(q^0+1+\sqrt{2q^0+2})}  +\theta\frac{|q|  }{2\sqrt{2q^0-2}}
	\\
	\notag
	=&- \frac{\theta|q| }{2\sqrt{2q^0-2}(2q^0+2)(q^0+1+\sqrt{2q^0+2})^2}
	\\
	\notag
	&\quad\quad\times\Bigg(\left(1+q^0\right)\left(1-q^0\right)\left(4\frac{1+q^0}{\sqrt{2q^0+2}}-2q^0+6\right)
	\\
	\notag
	&\quad\quad\quad -(2q^0+2)(q^0+1+\sqrt{2q^0+2})^2\Bigg)\\
	\notag
	&\ 	-\theta \frac{\sqrt{2q^0-2}|q| }{2\sqrt{2q^0+2}(q^0+1+\sqrt{2q^0+2})} .
	\end{align}
We continue to further calculte $P_2$ below, and in the last line we further note the important exact cancellation of the highest order $(q^0)^3$ terms, as
	\begin{align}
	\notag	
	P_2=&- \frac{\theta|q| }{2\sqrt{2q^0-2}(2q^0+2)(q^0+1+\sqrt{2q^0+2})^2}\\
	\notag	
	&\quad\quad\times\Bigg(\left(1-(q^0)^2\right)\left(4\frac{1+q^0}{\sqrt{2q^0+2}}+6\right)-2q^0(1-(q^0)^2)\\
	\notag	
	&\quad\quad\quad -2(q^0+1)^3 -(2q^0+2)(2q^0+2+2\sqrt{2q^0+2}(q^0+1))\Bigg)\\
	\notag	
	&\ 	-\theta \frac{\sqrt{2q^0-2}|q| }{2\sqrt{2q^0+2}(q^0+1+\sqrt{2q^0+2})} \\	
	\notag	
		=&- \frac{\theta|q| }{2\sqrt{2q^0-2}(2q^0+2)(q^0+1+\sqrt{2q^0+2})^2}\\
		\notag	
		&\quad\quad\times\Bigg(-|q|^2\left(4\frac{1+q^0}{\sqrt{2q^0+2}}+6\right)-2q^0-2(3(q^0)^2+3q^0+1)\\
		\notag	
	&\quad\quad\quad  -(2q^0+2)(2q^0+2+2\sqrt{2q^0+2}(q^0+1))\Bigg)\\
	&\ 	-\theta \frac{\sqrt{2q^0-2}|q| }{2\sqrt{2q^0+2}(q^0+1+\sqrt{2q^0+2})} .
	\notag	
\end{align} 
Here, we note that we find an exact cancellation to remove the highest order terms in $q^0$ to obtain the last identity, as the limit of $|q|\rightarrow \infty$ blows up otherwise. 

Then, we further compare the coefficients of the highest order terms in $|q|$ of the top and the bottom above.  Thus we obtain the following limit
\begin{equation}\label{P2lim}
	P_2 	\rightarrow  -\theta\frac{-4-4\sqrt{2}}{4\sqrt{2}}-\theta\frac{\sqrt{2}}{2\sqrt{2}}= \frac{1+\sqrt{2}}{2}\theta, \ \text{as} \ |q|\rightarrow \infty.
\end{equation}
We will use this limit at the end of the proof.
	
	Finally, we can observe from \eqref{P3} that $P_3$ would look like 
\begin{equation}\label{P3lim}\begin{split}
	P_3&=\theta^2\frac{\left(\gamma-1\right)^2}{4p^0\sqs |p+q|^4}\bigg(\sqs\left(\left(be+e^2\right)\left(p^0-q^0\right)\right)\\&\quad \quad+\left(1+p^0q^0-ad-be\right)\left(\sqs p^0-e^2+b^2+c^2\right)\bigg)\\
	&=\theta^2\frac{1}{4\sqrt{2q^0+2} (2q^0+2)(q^0+1+\sqrt{2q^0+2})^2} \bigg(\left(1+q^0\right)\sqrt{2q^0+2}\bigg)\\
	&=\theta^2 \frac{1}{8(q^0+1+\sqrt{2q^0+2})^2}\\
	&\rightarrow 0,\ \text{as} \ |q|\rightarrow \infty.
\end{split}
\end{equation}
Therefore, we conclude from \eqref{Alim}, \eqref{P2lim}, and \eqref{P3lim} that
\begin{equation}\label{Jqlim}\lim_{|q|\rightarrow \infty } \left|\det\left(\frac{\partial u}{\partial p}\right)\right|_{\text{at }p=0 \text{ and } q=-|q|w}=(1-\theta)^3+ \frac{1+\sqrt{2}}{2}\theta(1-\theta)^2.\end{equation}
 Finally, we take the limit as $\theta\rightarrow 1$ in \eqref{Jqlim} to finish the proof.
\end{proof}

This completes our discussion of the specific limit where the Jacobian in \eqref{det.up} can go to zero.  In the next section we will explain the results of our numerical study where we have seen that this Jacobian determiant in fact has a large number of distinct zeros.  

\section{Numerical Investigation of the Jacobian}\label{sec:numerical}

Understanding the roots of the Jacobian \eqref{det.up} provides us with information about the existence of solutions to the relativistic Boltzmann equation. We used numerical techniques to gain some understanding of the zeroes of the Jacobian. Random sampling of the domain showed that most points gave $\det \left( \frac{\partial u}{\partial p} \right) > 0$. Accordingly, we reduce the problem to that of finding $(\theta, p, q, w)$ which make the determinant negative. Following this, we pick another point with positive determinant and perform the bisection method along the path between these points. This allows us to obtain zeroes of arbitrary precision, relatively quickly. 

Examining the equations for the Jacobian in \eqref{det.up}, one can show that $\frac{\partial u}{\partial p_j}$ is continuous away from $p=q$. In general, there is a jump discontinuity along $p=q$, but this does not hinder the bisection method. Consider two points $\alpha = (\theta, p, q, w)$ and $\beta = (\theta', p', q', w')$. Let $\gamma(t) = (\theta(t), p(t), q(t), w(t))$ be the path from $\alpha$ to $\beta$. Then this path intersects the set $\{p=q\}$ if and only if $p(t) = q(t)$ for some $t$. Geometrically, if we plot $p(t)$ and $q(t)$ in $\R^3$, then this occurs if and only if these line segments intersect. This occurs with probability 0 and so in general the bisection method never encounters the jump discontinuity on $\{p=q\}$. Therefore the bisection method converges almost surely.

The descent algorithm chosen is random search. The algorithm begins by making an initial guess $(\theta, p, q, w)$. While $\det \left( \frac{\partial u}{\partial p} \right)(\theta, p, q, w) > 0$, a new point $(\theta, p', q', w')$ is chosen randomly in some ball about $(\theta, p, q, w)$. If $\det \left( \frac{\partial u}{\partial p} \right)(\theta, p', q', w') < \det \left( \frac{\partial u}{\partial p} \right)(\theta, p, q, w)$, then the guess is updated by setting $p = p', q=q', w=w'$. This method terminates after a negative determinant is found, or 100,000 iterations pass. Random search performed better than random guessing, particularly for small $\theta$. It did not find zeroes when $\theta \leq .1$. Random sampling of 200,000 points for $\theta \leq .1$ also failed to find negative values of the determinant. 

The script was written in SageMath and run on the General Purpose Cluster at the University of Pennsylvania. In order to guarantee high precision, we set the precision to 200 bits. Following this, we iterated over $\theta \in \{.01, .02, ..., .99\}$ and left $\theta$ fixed during the random search and bisection method. Before implementing this, the search algorithm tended to converge to values of $\theta$ close to 1. For each $\theta$, we performed the random search 50 times. After obtaining points $(\theta, p, q, w)$ and $(\theta, p', q', w')$ with determinants of opposite sign, we used the bisection method on these points and 49 other randomly generated points. This is done to obtain more data as the random search is computationally expensive. Finally, zeroes that do not satisfy the angle condition in \eqref{angle.condition} with \eqref{scattering.angle} are removed from the data set.   The written code used to run this algorithm is contained in \cite{chapman2020thesis}.

\begin{figure}[htbp]
\centering
\includegraphics[scale = .33]{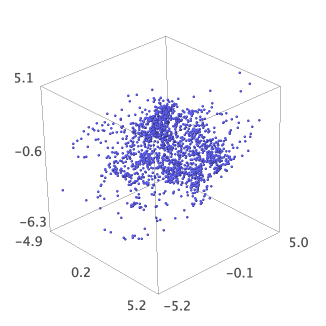}
\includegraphics[scale = .33]{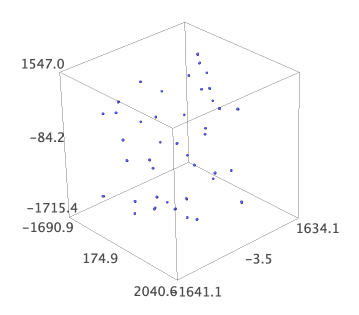}
\includegraphics[scale = .33]{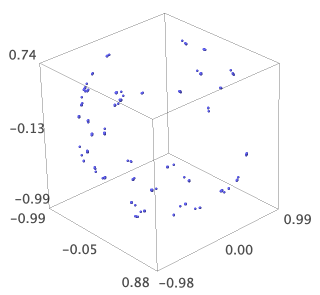}
\caption{3d plots of p, q, w, resp. (first 2000 points)}
\label{fig:PQWplot}
\end{figure}
Figure \ref{fig:PQWplot} shows plots of the roots. The plot for $p$ is very scattered, while the plots for $q$ and $w$ appear to be relatively ordered. This is due to the fact that only the first 2000 points are plotted, which corresponds to smaller values of $\theta$ in the data. This pattern does not hold for larger theta as can be seen in Figure \ref{fig:plots} where the data is very scattered. 
\begin{figure}[htbp]
\centering
\includegraphics[scale = .33]{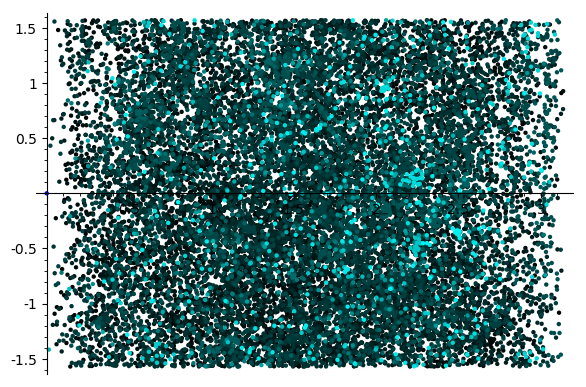}
\includegraphics[scale = .33]{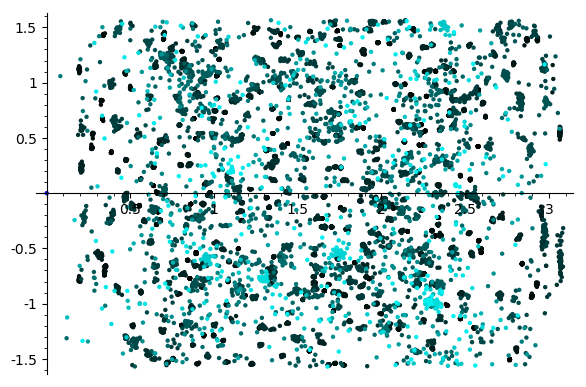}
\includegraphics[scale = .33]{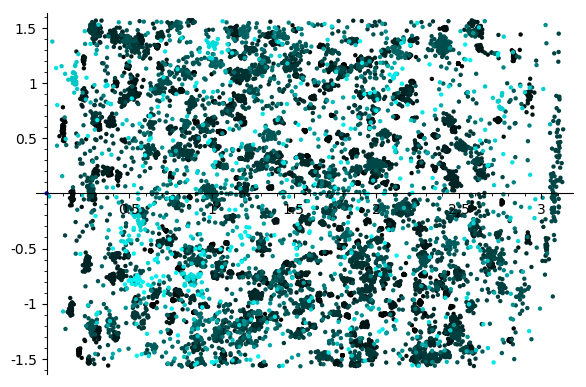}
\caption{These plots are obtained by taking $p, q, w$, resp., expressing them in spherical coordinates, and plotting the angular parts. Lighter color corresponds to larger $\theta$.}
\label{fig:plots}
\end{figure}

Examining plots of the data did not show any clear patterns. Figure \ref{fig:plots} shows that the zeros do not have a clear dependence on angle viewed independent of one another. In particular, the plot for $p$ displays roots of all angles and appears independent of $\theta$. The plots for $q$ and $w$ in Figure \ref{fig:plots} show more complexity and dependence on $\theta$, but there are no clear patterns in the plot.

	\appendix
	\section{An alternative representation of the Jacobian}\label{sec:ALTjacobian}
	
	In this appendix, we derive an alternative expression for the Jacobian determinant from the one in \eqref{det.up}.  For the expression in \eqref{det.up}, we remark that $P_2$ in \eqref{P2} and $P_3$ in \eqref{P3} are not independent of $A$ in \eqref{def.A} because both still contain the term $\frac{g\left(a+d\right)\left(\gamma-1\right)}{|p+q|^2}$ which is equal to $\frac{2}{\theta}\left(A-1+\frac{\theta}{2}\right)$. Define
	\begin{equation}\label{def.K}
	    K \eqdef \frac{g\left(a+d\right)\left(\gamma-1\right)}{|p+q|^2}=\frac{\left(\left(p+q\right)\cdot  w \right)g}{\left(\pqs\right)\sqs}.
	\end{equation}
The last calculation follows from \eqref{gamma.calc}, \eqref{def.abc} and \eqref{def.de}.
Therefore, as in \eqref{ineq.A}, $|K|$ is bounded by 1. Then we can write the Jacobian as a cubic polynomial in $K$ as in the following proposition:

	\begin{proposition} The Jacobian determinant can also be written as
		\label{Jaco2}
		$$\det\left( \frac{\partial u}{\partial p}\right)=D_1K^3+D_2K^2+D_3K+D_4,$$ where $D_i$ for $i=1,2,3,4$ is a function of $p$,$q$,$w$, and $\theta,$ which are defined explicitly in \eqref{Krepresentation} below.
	\end{proposition}

 \begin{proof}[Proof of Proposition \ref{Jaco2}]
	We rewrite the coefficient $A$ in \eqref{def.A} and the coefficients $P_2$ and $P_3$ in \eqref{P2} and \eqref{P3} respectively in terms of $K$ as
	\begin{equation*}
		A=\left(1-\frac{\theta}{2}\right)+\frac{\theta}{2}K.
	\end{equation*}
	We now rewrite $P_2$ in terms of $K$. We obtain directly from \eqref{P2} with \eqref{def.K} that
\begin{align*}
	P_2=&K\theta\bigg(\frac{\left(\gamma-1\right)}{2|p+q|^2p^0g^2}\{\left(1+p^0q^0-ad-be\right)\left(p^0-q^0\right)\left(4\gamma+4-g^2\right)\\
	&-\left(ad+be+|p|^2\right)g^2\sqs\}+\frac{1}{2}\bigg)+\theta\pqhat\\
	\eqdef&P_{21}K+P_{22}.
\end{align*}	
Similarly, we use \eqref{P3.appendix} to calculate for $P_3$ that
	\begin{align*}
	P_3=&\theta^2\frac{\left(a+d\right)\left(\gamma-1\right)^2\sqs}{4p^0|p+q|^4}\left(abe-bde+ae^2-c^2d-b^2d\right)\\
	&\hspace*{7mm}+\theta^2\frac{\left(\gamma-1\right)}{4p^0|p+q|^2}\left(q^0\left(-be-e^2\right)+p^0\left(b^2+c^2+be\right)\right).
	\end{align*} Then, from \eqref{def.K}, we have
	\begin{align*}
	P_3	&=\theta^2\frac{\left(\gamma-1\right)\sqs}{4p^0|p+q|^2g}\bigg(K\left(abe-bde-ae^2-c^2d-b^2d\right)\\
	&\qquad+\frac{g}{\sqs}\{q^0\left(-be-e^2\right)+p^0\left(b^2+c^2+be\right)\}\bigg)
	\\
	&
	\eqdef P_{31}K+P_{32}.
	\end{align*}
	Now, using the simplifications above, the determinant from \eqref{det.up} is 
\begin{multline}\label{Krepresentation}
	\det\left(\mathbf{\Phi}\right)= A^3+A^2P_2+AP_3
	\\
	=\left(\left(1-\frac{\theta}{2}\right)+\frac{\theta}{2}K\right)^3+\left(\left(1-\frac{\theta}{2}\right)+\frac{\theta}{2}K\right)^2\left(P_{21}K+P_{22}\right)
	\\
+\left(\left(1-\frac{\theta}{2}\right)+\frac{\theta}{2}K\right)\left(P_{31}K+P_{32}\right)
	\\
	=\left(\left(\tet\right)^3+P_{21}\left(\tet\right)^2\right)K^3
	\\
	+\bigg(
	\tet\left(1-\tet\right)\left(\frac{3\theta}{2}+2P_{21}\right)
	+\tet\left(P_{22}\left(\tet\right)+P_{31}\right)\bigg)K^2
	\\
	+\bigg(3\left(1-\tet\right)^2\tet+\left(1-\tet\right)^2P_{21}
	+\theta\left(1-\tet\right)P_{22}+\left(1-\tet\right)P_{31}+\tet P_{32}\bigg)K
	\\
	+\left(\left(1-\tet\right)^3+P_{22}\left(1-\tet\right)^2+P_{32}\left(1-\tet\right) \right)
	\\
	\eqdef D_1K^3+D_2K^2+D_3K+D_4.
\end{multline}
This completes the proof.
	\end{proof}

\noindent{\bf Acknowledgements} J. W. Jang was supported by the DFG grant CRC 1060 of Germany and was partially supported by the NSF grant DMS-1500916 of the USA.  
R. M. Strain was partially supported by the NSF grant DMS-1764177 of the USA.

\providecommand{\bysame}{\leavevmode\hbox to3em{\hrulefill}\thinspace}
\providecommand{\href}[2]{#2}


\begin{thebibliography}{10}
\expandafter\ifx\csname arxiv\endcsname\relax
  \def\arxiv#1{\burlalt{arXiv:#1}{http://arxiv.org/abs/#1}}\fi
\expandafter\ifx\csname doi\endcsname\relax
  \def\doi#1{\burlalt{doi:#1}{http://dx.doi.org/#1}}\fi
\expandafter\ifx\csname href\endcsname\relax
  \def\href#1#2{#2}\fi
\expandafter\ifx\csname burlalt\endcsname\relax
  \def\burlalt#1#2{\href{#2}{#1}}\fi

\bibitem{ADVW}
R.~Alexandre, L.~Desvillettes, C.~Villani, and B.~Wennberg, \emph{Entropy
  dissipation and long-range interactions}, Arch. Ration. Mech. Anal.
  \textbf{152} (2000), no.~4, 327--355, \doi{10.1007/s002050000083}.

\bibitem{MR2795331}
R.~Alexandre, Y.~Morimoto, S.~Ukai, C.-J. Xu, and T.~Yang, \emph{Global
  existence and full regularity of the {B}oltzmann equation without angular
  cutoff}, Comm. Math. Phys. \textbf{304} (2011), no.~2, 513--581,
  \doi{10.1007/s00220-011-1242-9}.

\bibitem{C-K}
Carlo Cercignani and Gilberto~Medeiros Kremer, \emph{The relativistic
  {B}oltzmann equation: theory and applications}, Progress in Mathematical
  Physics, vol.~22, Birkh\"auser Verlag, Basel, 2002,
  \doi{10.1007/978-3-0348-8165-4}.

\bibitem{chapman2020thesis}
James Chapman, \emph{Numerical study of a {J}acobian determinant for the
  relativistic {B}oltzmann equation}, Master's thesis, University of
  Pennsylvania, Department of Mathematics, Philadelphia, USA, 2020.

\bibitem{DeGroot}
S.~R. de~Groot, W.~A. van Leeuwen, and Ch.~G. van Weert, \emph{Relativistic
  {K}inetic {T}heory. principles and applications.}, North-Holland Publishing
  Co., Amsterdam-New York, 1980.

\bibitem{GS4}
R.~T. Glassey and W.~A. Strauss, \emph{Asymptotic stability of the relativistic
  {M}axwellian via fourteen moments}, Transport Theory Statist. Phys.
  \textbf{24} (1995), no.~4-5, 657--678, \doi{10.1080/00411459508206020}.

\bibitem{GL1996}
Robert~T. Glassey, \emph{The {C}auchy problem in kinetic theory}, Society for
  Industrial and Applied Mathematics (SIAM), Philadelphia, PA, 1996,
  \doi{10.1137/1.9781611971477}.

\bibitem{MR1379589}
\bysame, \emph{The {C}auchy problem in kinetic theory}, Society for Industrial
  and Applied Mathematics (SIAM), Philadelphia, PA, 1996,
  \doi{10.1137/1.9781611971477}.

\bibitem{MR1105532}
Robert~T. Glassey and Walter~A. Strauss, \emph{On the derivatives of the
  collision map of relativistic particles}, Transport Theory Statist. Phys.
  \textbf{20} (1991), no.~1, 55--68, \doi{10.1080/00411459108204708}.

\bibitem{GS3}
\bysame, \emph{Asymptotic stability of the relativistic {M}axwellian}, Publ.
  Res. Inst. Math. Sci. \textbf{29} (1993), no.~2, 301--347,
  \doi{10.2977/prims/1195167275}.

\bibitem{Gra}
Harold Grad, \emph{Asymptotic theory of the {B}oltzmann equation. {II}},
  Rarefied {G}as {D}ynamics ({P}roc. 3rd {I}nternat. {S}ympos., {P}alais de
  l'{UNESCO}, {P}aris, 1962), {V}ol. {I}, Academic Press, New York, 1963,
  pp.~26--59.

\bibitem{MR2784329}
Philip~T. Gressman and Robert~M. Strain, \emph{Global classical solutions of
  the {B}oltzmann equation without angular cut-off}, J. Amer. Math. Soc.
  \textbf{24} (2011), no.~3, 771--847, \arxiv{1011.5441},
  \doi{10.1090/S0894-0347-2011-00697-8}.

\bibitem{Guo-Strain2}
Yan Guo and Robert~M. Strain, \emph{Momentum regularity and stability of the
  relativistic {V}lasov-{M}axwell-{B}oltzmann system}, Comm. Math. Phys.
  \textbf{310} (2012), no.~3, 649--673, \doi{10.1007/s00220-012-1417-z}.

\bibitem{Jang2016}
Jin~Woo Jang, \emph{{G}lobal classical solutions to the relativistic
  {B}oltzmann equation without angular cut-off}, Ph.D. thesis, University of
  Pennsylvania, 2016, (ProQuest Document ID 1802787346), pp.~1--135.

\bibitem{1907.05784}
Jin~Woo Jang, Robert~M. Strain, and Seok-Bae Yun, \emph{Propagation of uniform
  upper bounds for the spatially homogeneous relativistic {B}oltzmann
  equation}, preprint (2019), 1--31, \arxiv{1907.05784}.

\bibitem{MR2728733}
Robert~M. Strain, \emph{Asymptotic stability of the relativistic {B}oltzmann
  equation for the soft potentials}, Comm. Math. Phys. \textbf{300} (2010),
  no.~2, 529--597, \arxiv{1003.4893}, \doi{10.1007/s00220-010-1129-1}.

\bibitem{MR2679588}
\bysame, \emph{Global {N}ewtonian limit for the relativistic {B}oltzmann
  equation near vacuum}, SIAM J. Math. Anal. \textbf{42} (2010), no.~4,
  1568--1601, \arxiv{1004.5407}, \doi{10.1137/090762695}.

\bibitem{MR2765751}
\bysame, \emph{Coordinates in the relativistic {B}oltzmann theory}, Kinet.
  Relat. Models \textbf{4} (2011), no.~1, 345--359, \arxiv{1011.5093},
  \doi{10.3934/krm.2011.4.345}.

\bibitem{1903.05301}
Robert~M. Strain and Zhenfu Wang, \emph{Uniqueness of bounded solutions for the
  homogeneous relativistic {L}andau equation with {C}oulomb interactions},
  Quart. Appl. Math. \textbf{78} (2020), 107--145, \arxiv{1903.05301},
  \doi{10.1090/qam/1545}.

\end{thebibliography}
\end{document}